\documentclass[10pt,reqno, english]{article}
\usepackage{latexsym, amsmath, amssymb, a4, epsfig}
\usepackage{graphicx}
\usepackage{color}

\usepackage{esint}

\makeatletter
\newcommand{\ds}{\displaystyle}
\newcommand{\nm}{\noalign{\smallskip}}

\numberwithin{equation}{section}
\numberwithin{figure}{section}
\numberwithin{table}{section}

\usepackage{babel}

\usepackage{amsthm}

\usepackage{babel}
\newtheorem{thm}{Theorem}
\newtheorem{lem}{Lemma}
\newtheorem{prop}{Proposition}

\makeatother

\usepackage{babel}

\begin{document}

\title{Modeling active electrolocation in weakly electric fish\thanks{\footnotesize This work
was supported  by ERC Advanced Grant Project MULTIMOD--267184.}}

\author{Habib Ammari\thanks{\footnotesize Department of Mathematics and Applications,
Ecole Normale Sup\'erieure, 45 Rue d'Ulm, 75005 Paris, France
(habib.ammari@ens.fr, boulier@dma.ens.fr).} \and Thomas
Boulier\footnotemark[2] \and Josselin Garnier\thanks{\footnotesize
Laboratoire de Probabilit\'es et Mod\`eles Al\'eatoires \&
Laboratoire Jacques-Louis Lions, Universit\'e Paris VII, 75205
Paris Cedex 13, France (garnier@math.jussieu.fr).}}\maketitle

\begin{abstract}
In this paper, we provide a mathematical model for the
electrolocation in weakly electric fishes. We first investigate
the forward complex conductivity problem and derive the
approximate boundary conditions on the skin of the fish. Then we
provide a dipole approximation for small targets away from the
fish. Based on this approximation, we obtain a non-iterative
location search algorithm using multi-frequency measurements. We
present numerical experiments to illustrate the performance and
the stability of the proposed multi-frequency location search
algorithm. Finally, in the case of disk- and ellipse-shaped
targets, we provide a method to reconstruct separately the
conductivity, the permittivity, and the size of the targets from
multi-frequency measurements.
\end{abstract}



\section{Introduction}

In the turbid rivers of Africa and South America, some species of
fish generate an electric current which is not enough for defense
purpose. In 1958, Lissmann and Machin discover that this electric
current is in fact used for spatial visualization
\cite{lissmann1958mechanism}. Indeed an object in the vicinity of
the fish will be detected by measurement of the electric field's
distortion on the skin. Behavioral experiments have shown that the
weakly electric fish is able to extract useful information about
targets, such as the location \cite{von1993electric}, the shape
\cite{von2007distance}, and the electric parameters (capacitance
and conductivity) \cite{von1999active}.

Mathematically speaking, this is an inverse problem for the
electric field created by the fish. Indeed, given the current
distribution over the skin,  the problem is to recover the
conductivity distribution in the surrounding space. Due to the
ill-posedness of this type of problems, it is very difficult to
recover as much information as the fish is able to. Thus,
modelling this ``electric sense'' (called active electrolocation)
is likely to give us insights in this regard.

Electrolocation has been quantitatively investigated since
Lissmann and Machin, who tried an analytical approach. More
precisely, they computed the distortion created by a cylinder
placed in the electric field of a dipole
\cite{lissmann1958mechanism}, and noticed that it is equivalent to
the field created by a dipole located inside the cylinder. In
1983, Bacher remarked that this formula cannot explain the phase
difference observed when the electric permittivity of the target
does not equal the permittivity of the water \cite{bacher1983new}.
This phase shift seems to be an important input for the fish since
it is measured by receptors (called Rapid Timing units
\cite{moller1995}), and thus will be the central point in this
paper. Rasnow in 1996 gathered these two previous results by
considering a time-harmonic and spatially uniform background
electric field. In the presence of a sphere with center at $0$ and
radius $a$, conductivity $\sigma_{1}$, and permittivity
$\varepsilon_{1}$, the uniform background electric field ${E}_{0}$
with frequency $\omega$ is modified by adding the following field:
\begin{equation} \label{formemp}
{E}_{0}\cdot {x} = \left(\frac{a}{|{
x}|}\right)^{3}\frac{(\sigma_{1}+i\omega\varepsilon_{1})-(\sigma_{0}
+i\omega\varepsilon_{0})}{2(\sigma_{1}+i\omega\varepsilon_{1})+(\sigma_{0}+i\omega\varepsilon_{0})},
\end{equation}
 where the index $0$ refers to the ambient medium.

Numerical approaches have also been driven since the 70's: in
1975, Heiligenberg proposed a finite differences scheme to
calculate the field created by the fish
\cite{heiligenberg1975theoretical}. In 1980, Hoshimiya et al. use
finite elements to solve this problem
\cite{hoshimiya1980theapteronotus}. The geometry of the fish is
simplified by an ellipse and is divided into two areas: the thin
skin with low conductivity and the interior of the body. Their aim
is to optimize conductivity values to approximate as better as
possible the experimentally measured field. The result is that the
optimal conductivity is non-uniform, being higher in the tail
region. Improvements of these models since then can be found in
\cite{babineau2006modeling,maciver2001computational,migliaro2005theoretical,rasnow1989simulation}
and references therein. However, the most promising technique is
the use of the boundary element method performed by Assad in the
90's in his PhD thesis \cite{assad1997electric}. Indeed, the
important feature is the electric potential on the skin (because
it is the input for the fish), so a boundary element method (BEM)
approach allows us to concentrate the equations on it. Moreover,
the computation speed is enhanced because the number of nodes is
dramatically reduced. The equation considered is here $\Delta u=0$
on the exterior of the body with Robin boundary conditions on the
skin \cite{assad1990hypercube}:
\begin{equation}
u-\xi\frac{\partial u}{\partial\nu}=\psi,\label{eq:assad_BC}
\end{equation}
 where $\psi$ is the potential inside the body and $\xi=h(\sigma_{0}/\sigma_{s})$
($h$ being the skin thickness and $\sigma_{s}$ (resp. $\sigma_{0}$)
the skin (resp. water) conductivity) is the \emph{effective skin thickness}.

Let us mention that there are other kinds of simulations, based on
a more empirical approach, determining an equivalent electric
circuit \cite{budelli2000electric,caputi1998electric} or an
equivalent multipole \cite{chen2005modeling}.

The aim of this paper is to derive a rigorous model for the
electrolocation of an object around the fish. Two problems arise:
the direct problem, \emph{i.e.}, the equations involved and their
boundary conditions, and the reconstruction itself. For the direct
complex conductivity problem, we show using layer potential
techniques the validity of (\ref{eq:assad_BC}). We also generalize
formula (\ref{formemp}) to the case of a non-uniform background
electric field, taking into account the distortion induced by the
body of the fish, and with any shape of the target. For the
inverse problem, little is known in the complex conductivity case
\cite{beretta}. Here, we take advantage of the smallness of the
targets to use the framework of small volume asymptotic expansions
for target location and characterization
\cite{ammari2004reconstruction, ammari2007polarization}. However,
since the electric current is generated by only one emitter at the
tail of the fish (the electric organ) and measured by many
receptors on the skin, standard non-iterative algorithms such as
MUSIC (standing for MUltiple Signal Classification) cannot be
applied for location search. In standard MUSIC, the data (called
multistatic response matrix) form a matrix and its singular value
decomposition leads to an efficient imaging function by projecting
the Green function of the medium onto the significant image space
\cite{park,ail,ammarinumer, bruhl2003direct, chambers, cheney,
kirsch}. Here, roughly speaking, one has only a column of the
response matrix. However, using the fact that the electric current
produced by the electric organ is periodically time dependent with
a known fundamental frequency, we extend MUSIC approach to
multi-frequency measurements by constructing an efficient and
robust multi-frequency MUSIC imaging function. We perform
numerical simulations in order to validate both the direct model
and the multi-frequency MUSIC algorithm. We also illustrate the
robustness with respect to measurement noise and the sensitivity
with respect to the number of frequencies, the number of sensors,
and the distance to the target of the location search algorithm.
Finally, in the case of disk- and ellipse-shaped targets, we
provide a method to reconstruct separately the conductivity, the
permittivity, and the size of the targets from multi-frequency
measurements. We mention that this is possible only because of
multi-frequency measurements which yield polarization tensors with
complex conductivities. It is well-known that polarization tensors
for real conductivities cannot separate the size from material
properties of the target \cite{ammari2007polarization}. We also
mention that the use of different values for the frequencies is
more crucial for the material and size reconstruction procedure
than for the location step. In fact, in the presence of
measurement noise, location with $N$ realizations with one
frequency is comparable to the one with $N$ different frequency
values.

The paper is organized as follows. In section
\ref{sec:forward_problem}, the model is set up and the equations
governing the electric field are rigorously derived. Using layer
potential techniques, the boundary condition (\ref{eq:assad_BC})
is recovered. In section \ref{sec:detection_algo}, a small target
is located using multi-frequency measurements. For this purpose, a
dipolar approximation is derived before the analysis of the
response matrix. Finally, numerical simulations are performed in
section \ref{sec:numeric}; due to the presence of a hyper-singular
operator, a particular attention is paid to the numerical scheme.
Reconstructions of the electromagnetic parameters and the size of
disk- and ellipse-shaped targets are also provided.

\section{The forward problem}

\label{sec:forward_problem}

The aim of this section is to formulate the \emph{forward
problem}. After the setup of the problem in subsection
\ref{sub:setup}, the boundary conditions are announced in
subsection \ref{sub:BC-announce} before being derived in
subsection \ref{sub:BC-derivation}. Existence, uniqueness and a
useful representation lemma for this derivation are proved in
subsection \ref{sub:existence-uniqueness}.

\subsection{Non-dimensionalization and problem formulation}

\label{sub:setup}

In this subsection, we derive the equations governing the electric
field. A formal explanation of the electroquasistatic (or EQS) formulation
is given, and the setup of the problem is then non-dimensionalized.

\subsubsection*{Partial differential equations of the problem}

The electroquasistatic (or EQS) formulation is a low-frequency
limit for the Maxwell system in three dimensions. In the frequency
domain, this latter is given by
\begin{equation}
\left\{ \begin{alignedat}{1}\nabla\cdot\varepsilon {E} & =\rho,\\
\nabla\cdot {B} & =0,\\
\nabla\times {E} & =-i\omega {B},\\
\nabla\times\frac{ {B}}{\mu} & = {j}+i\omega\varepsilon {E},
\end{alignedat}
\right.\label{eq:maxwell}
\end{equation}
where $E$ is the electric field, $B$ is the magnetic induction
field, $\rho$ and $j$ are the free charges and currents, $\omega$
is the frequency, $\mu$ is the magnetic permeability, and
$\varepsilon$ is the electric permittivity. Moreover, in a medium
of conductivity $\sigma$ the Ohm's law connects the electric field
to the induced current density (${j_{i}}=\sigma {E}$) so the
current density can be decomposed as:
\[
{j}=\sigma {E}+ {j_{s}},
\]
 where ${j_{s}}$ is a source of current (in our model, it
comes from the electric organ). Then, taking the divergence of the
last line in (\ref{eq:maxwell}), we have:
\begin{equation}
\nabla\cdot(\sigma+i\varepsilon\omega) {E}=-\nabla\cdot
{j_{s}}.\label{eq:div-maxwell4}
\end{equation}
The EQS approximation consists in considering the electric field
as irrotational because the magnetic field variation is
negligible. A sufficient condition for that is given by
\cite{vanRienen2001}:
\begin{equation}
\frac{L_{\rm max}}{\lambda_{\rm min}}\ll1,\label{eq:eqs_condition}
\end{equation}
 where $L_{\rm max}$ is the maximal length of the problem and $\lambda_{\rm min}$
the minimal wavelength. Here, we can take $L_{\rm max}=1$m because
the range of electrolocation does not exceed two body lengths
\cite{moller1995}. In the water, the minimal wavelength is given
by
\[
\lambda_{\rm min}=\frac{1}{\omega_{\rm max}\sqrt{\mu\varepsilon}},
\]
 where $\mu\approx\mu_{0}$, $\varepsilon\approx80\varepsilon_{0}$
and $\omega_{\rm max}$ is the maximal frequency emitted by the
fish, which is of the order of $10$kHz. Thus, the fraction in
(\ref{eq:eqs_condition}) is of order $10^{-4}$, so the EQS
approximation is very well suited for our situation.

Going back to the equation of the electric field
(\ref{eq:div-maxwell4}), we can now use the fact that ${E}$ is
irrotational to state that it is derived from a potential scalar
field $u$. This finally leads us to the following equation:
\begin{equation}
\nabla\cdot(\sigma+i\varepsilon\omega)\nabla u=-\nabla\cdot
{j_{s}}.\label{eq:EQS-PDE}
\end{equation}

To conclude, taking into account the slow variation of the
electric field leads us to consider a complex conductivity instead
of a real valued one. However, for the rest of this section, the
imaginary part of this conductivity will be neglected; indeed
measurements on a \emph{Gnathonemus petersii} showed that the
permittivity of the skin, the body, and the water are very small
compared to their respective
conductivity~\cite{caputi1998electric,scheich1973coding}. Thus,
this EQS approximation will be used only in the presence of a
target: it will be detected by the phase shift induced by its
complex conductivity.

\subsubsection*{Non-dimensionalization}

We wish to perform an asymptotic analysis of the equations. The
first step consists in the identification of the different scales
of the model problem. The electric potential $u$, the variables
$x$ and $\omega$, and the parameters $\sigma$ and ${j_{s}}$ can be
written as follows:
\[
u=V_{0}u',\;
x=Lx',\;\omega=\omega_{0}\omega',\;\sigma=\sigma_{0}k,\;
{j_{s}}=\frac{I_{0}}{L^{2}} {j_{s}'},
\]
 where $V_{0}$ is the voltage produced by an \emph{electric organ
discharge} (EOD), $L$ is the length of the fish, $\omega_{0}$ is
the fundamental frequency of the EOD, $\sigma_{0}$ is the
conductivity of the surrounding water and $I_{0}$ is the current
intensity inside the electric organ. Moreover, anticipating the
next subsection, the conductivity of the body and the skin play an
important role in the shape of the electric field. Thus, in the
list of parameters we add the conductivity of the body
$\sigma_{b}$, the thickness of the skin $h$ and its surface
conductivity $\Sigma$. The orders of magnitude of these parameters
are found in Table~\ref{tab:Orders-of-magnitude}.

\begin{table}[!h]
\centering%
\begin{tabular}{|c|c|c|}
\hline
Quantity  & Order of magnitude  & Reference\tabularnewline
\hline
\hline
$V_{0}$  & $10$ mV  & \cite{assad1998electric,stoddard1999electric}\tabularnewline
\hline
$L$  & $10$ cm  & \cite{moller1995}\tabularnewline
\hline
$\omega_{0}$  & $1$ kHz  & \cite{moller1995}\tabularnewline
\hline
$\sigma_{0}$  & $100$ $\mu$S$\cdot$cm$^{-1}$  & \cite{maciver2001prey}\tabularnewline
\hline
$I_{0}$  & $1$ mA  & \cite{bell1976electric}\tabularnewline
\hline
$\sigma_{b}$  & $1$ S$\cdot$m$^{-1}$  & \cite{scheich1973coding}\tabularnewline
\hline
$\Sigma$  & $100$ $\mu$S$\cdot$cm$^{-2}$  & \cite{caputi1998electric}\tabularnewline
\hline
$h$  & $100$ $\mu$m  & \cite{zakon1986electroreceptive}\tabularnewline
\hline
\end{tabular}

\caption{\label{tab:Orders-of-magnitude}Orders of magnitude of the
physical quantities involved. These are only scales and not the
exact values measured in the cited references. Here $S$ is Siemens
($1S =1A/1V$).}
\end{table}

These $n=8$ quantities involve $r=4$ fundamental units of the SI
system, so according to the Buckingham-Pi theorem, we need $n-r=4$
nondimensional quantities. The first one can be found by rewriting
the equation (\ref{eq:EQS-PDE}) in terms of the nondimensional
quantities ($x^\prime, k, u', j_{s}^\prime$):
\begin{equation}
\nabla_{x'}\cdot
k\nabla_{x'}u'=-\frac{I_{0}}{\sigma_{0}V_{0}L}\nabla\cdot
{j_{s}'}.\label{eq:nondimensionalized-EQS}
\end{equation}

The multiplicative term in the right-hand side of the previous
equation is not important as the equation is linear. The three
other nondimensional quantities come from the parameters of the
skin and the body of the fish:
\[
k_{b}:=\frac{\sigma_{b}}{\sigma_{0}}\sim10^{2},\; k_{s}:=\frac{h\Sigma}{\sigma_{0}}\sim10^{-2},\;\delta:=\frac{h}{L}\sim10^{-3}.
\]
 In other words, in nondimensional units, $k_{b}$ (resp. $k_{s}$)
is the body (resp. skin) conductivity and $\delta$ is the skin thickness.

To conclude, omitting the prime symbol for the sake of clarity and
denoting by $k_b f$ the source term in equation
(\ref{eq:nondimensionalized-EQS}), the governing PDE is the
following
\begin{equation}
\nabla\cdot k\nabla u= k_b f,\label{eq:EQS-final}
\end{equation}
 where $k$ is piecewise constant, being equal to $1$ in the water,
$k_{b}$ inside the body of the fish and $k_{s}$ in the skin. These
domains are going to be made precise in the next subsection.

For the sake of simplicity, from now on, we only consider  the
model equations in two dimensions.

\subsection{Boundary conditions}

In this subsection, we derive the appropriate boundary conditions
associated with the presence of a very thin and very resistive
skin. Robin boundary conditions will be found after an asymptotic
analysis of the layer potentials involved.

The setup is as follows: the body occupies a fixed smooth open set
$\Omega_{b}$ and the skin with constant thickness is described as:
\[
\Omega_{s}:=\big\{ x +t\nu(x),\,
x\in\partial\Omega_{b},\,0<t<\delta\big\},
\]
 where $\nu$ is the outward normal unit vector. Let us also denote
by $\xi$ the effective thickness defined by Assad
\cite{assad1990hypercube}; in our variables it is given by
\[
\xi:=\frac{\delta}{k_{s}}.
\]
 The source of the electric field is a sum of Dirac functions:
\[
f= \sum_{j=1}^{m}\alpha_{j}\delta_{z_{j}},
\]
 where, for $1\leq j\leq m$, $z_{j}\in\Omega_{b}$ and $f$ satisfies
 the charge neutrality condition
\begin{equation} \label{neutre}
\ds \sum_{j=1}^m \alpha_{j}=0.
\end{equation}
Although condition (\ref{neutre}) is the physical condition in our
model, we will show how to modify the derivations and the results
of the paper in the general case. An illustration is given in
Figure~\ref{fig:Setup}.
\begin{figure}
\centering\includegraphics[width=10cm]{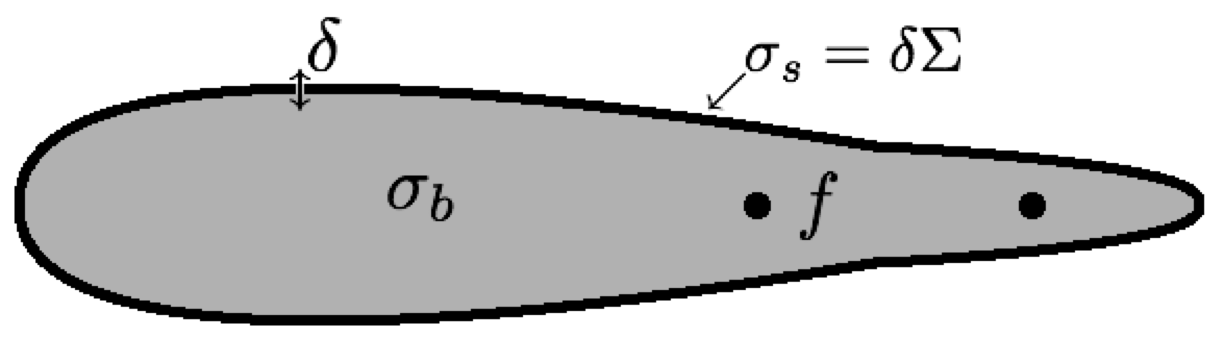} \caption{Setup of
the problem. The conductivities are non-dimensionalized so that
$\sigma_{0}=1$. The body $\Omega_{b}$ is represented in grey and
the skin $\Omega_{s}$ is represented by its bold boundary. The
sources $f$ are given by the two dots. \label{fig:Setup}}
\end{figure}

\label{sub:BC-announce}

Our main purpose here is to investigate the behavior of the
solution of (\ref{eq:EQS-final}) with
\begin{equation} \label{defk}
k(x)=\left\{\begin{array}{l} k_s \quad \mbox{if } x \in \Omega_s,\\
\nm k_b \quad \mbox{if } x \in \Omega_b,\\
\nm 1 \quad \mbox{otherwise},
\end{array}
\right.
\end{equation}
where $k_s \neq 1$ and $k_b \neq k_s$, in the following asymptotic
regime:
\[
k_s = \frac{\delta}{\xi}, \quad \xi \mbox{ is fixed}, \quad
\delta\rightarrow0, \mbox{ and }k_{b}\rightarrow\infty.
\]
 In order to make this dependence clear, let us denote such a solution by $u_{\delta,k_{b}}$. Adding a far field condition (essential for uniqueness, see
subsection \ref{sub:existence-uniqueness}), it is the solution of
\begin{equation}
\left\{ \begin{alignedat}{2}\nabla\cdot k\nabla u_{\delta,k_{b}} & = k_b f, & \,\, x\in\mathbb{R}^{2},\\
\left|u_{\delta,k_{b}} \right| & = {O}(\left|x\right|^{-1}), &
\,\,\left|x\right|\rightarrow\infty \text{ uniformly in }\hat{x},
\end{alignedat}
\right.\label{eq:governing_equation}
\end{equation}
 where $\hat{x}:=x/\left|x\right|$ and $k(x)$ is given by
 (\ref{defk}). Note that if assumption (\ref{neutre}) does not hold, then the far
field condition should be replaced with
\begin{equation} \label{eq:governing_equation_rad}
\left|u_{\delta,k_{b}} - (\sum_{j=1}^m \alpha_j)  \frac{(\lambda_b
+ 1/2) (\lambda_s + 1/2)}{2 \pi (\lambda_b -1/2) (\lambda_s -
1/2)} \log |x| \right|  = {O}(\left|x\right|^{-1}), \quad
\,\,\left|x\right|\rightarrow\infty \text{ uniformly in }\hat{x},
\end{equation}
where the parameters $\lambda_{s}$ and $\lambda_{b}$ are given by
\begin{equation} \label{deflambdas}
\lambda_{s}:=\frac{k_{s}+1}{2(k_{s}-1)}\mbox{ and
}\lambda_{b}:=\frac{k_{s}+k_{b}}{2(k_{s}-k_{b})}.
\end{equation}

The far field condition (\ref{eq:governing_equation_rad}) will be
explained later. We will compute the first-order asymptotic
$u_{0,\infty}$ and see that it is the solution of the following
system:
\begin{equation}
\left\{ \begin{alignedat}{2}\Delta u_{0,\infty} & ={f}, & \,\, x\in\Omega_{b},\\
\Delta u_{0,\infty} & =0, & \,\, x\in\mathbb{R}^{2}\setminus\overline{\Omega}_{b},\\
\left.u_{0,\infty}\right|_{+}-\left.u_{0,\infty}\right|_{-} & =\xi\left.\frac{\partial u_{0,\infty}}{\partial\nu}\right|_{+}, & \,\, x\in\partial\Omega_{b},\\
\left.\frac{\partial u_{0,\infty}}{\partial\nu}\right|_{-} & =0, & \,\, x\in \partial \Omega_{b},\\
\left|u_{0,\infty}\right| & = {O}(\left|x\right|^{-1}), &
\,\,\left|x\right|\rightarrow\infty,\text{ uniformly in }\hat{x}.
\end{alignedat}
\right.\label{eq:asymptotic_equation}
\end{equation}
Note that in the limiting model (\ref{eq:asymptotic_equation}),
the role of $f$ is to fix the potential $ u_{0,\infty} \big|_{-}$
on $\partial \Omega_{b}$. On the other hand, if assumption
(\ref{neutre}) does not hold, then the boundary condition on
$\frac{\partial u_{0,\infty}}{\partial\nu}\big|_{-}$ should be
replaced with
$$
\frac{\partial u_{0,\infty}}{\partial\nu}\big|_{-} =
\frac{1}{|\partial \Omega_b|} \sum_{j=1}^m \alpha_j.
$$

To be more precise, we will prove the following theorem:
\begin{thm} \label{thm:main-result}There exists a constant $C$
independent of $\delta$ and $k_{b}$ such that the following
inequality holds for $\delta$ and $1/k_{b}$ small enough:
\begin{equation}
\left\Vert u_{\delta,k_{b}}-u_{0,\infty}\right\Vert
_{L^{\infty}(\mathbb{R}^{2})}\leq
C\left(\delta+\frac{1}{k_{b}}\right),\label{eq:estimate-general}
\end{equation}
 where $u_{\delta,k_{b}}$ and $u_{0,\infty}$ are the solutions of (\ref{eq:governing_equation})
and (\ref{eq:asymptotic_equation}), respectively.
\end{thm}

In a first part, we will analyze
equation~(\ref{eq:governing_equation}) and show that there exists
a unique solution that can be represented as the sum of a harmonic
function and two single-layer potentials. In a second part, we
will perform asymptotic analysis of these layer potentials in
order to show that the limiting function is solution
of~(\ref{eq:asymptotic_equation}). This latter part is to adapt
the work done by Zribi in his thesis~\cite{zribilayer} and by
Zribi and Khelifi in \cite{khelifizribi}.

\subsubsection{Existence, uniqueness, and representation of the electric potential}

\label{sub:existence-uniqueness}

In this part, we will first prove  the uniqueness of the solutions
of~(\ref{eq:governing_equation}) and then we will derive a
representation formula, which will give us the existence of the
solution. For the moment, $\delta$ and $k_{b}$ are fixed, but we
suppose that:
\begin{equation}
k_{s} < 1 < k_{b}.\label{eq:ordering_conductivities}
\end{equation}

\subsubsection*{Uniqueness}

The uniqueness comes from the second line
of~(\ref{eq:governing_equation}) \cite{ammari2007polarization}.
Indeed, let $v=u_1-u_2$, where $u_1$ and $u_2$ are two solutions
of (\ref{eq:governing_equation}) and let us show that $v=0$.
From~(\ref{eq:ordering_conductivities}) we get, for $R$
sufficiently large (so that the ball with center $0$ and radius
$R$ encompasses $\Omega_s$):
\begin{align*}
\int_{\left|x\right|<R}\left|\nabla v\right|^{2}
\leq\frac{1}{k_{s}} \int_{\left|x\right|<R} k(x) \left|\nabla
v\right|^{2} =
 \frac{1}{k_{s}}\int_{\left|x\right|=R}v\frac{\partial v}{\partial\nu}=
 -\frac{1}{k_{s}}\int_{\left|x\right|>R}\left|\nabla
v\right|^{2}\leq0.
\end{align*}
Here we have used the fact that  $\nabla v \in
L^2(\mathbb{R}^{2}\setminus \overline{\Omega}_s)$, which holds as
a consequence of the far field condition.
 A unique continuation argument shows that
$\left|\nabla v\right|^{2}=0$ in $\mathbb{R}^{2}$ and thus $v$ is
constant. Then, using the fact that $v\rightarrow0$ as
$|x|\rightarrow \infty$, we have $v=0$.

\subsubsection*{Existence and representation}

The existence is given by a representation formula decomposing the
solution into a source part and a refraction part. This refraction
part implies layer potentials on the boundaries of the body and
the skin. Let us define them explicitly and give some well-known
results. First, let us define the following boundaries:
\[
\Gamma_{b}:=\partial\Omega_{b}\mbox{ and }\Gamma_{s}:=\partial\Omega_{s}\setminus\Gamma_{b}.
\]
 In the following, the index $\beta$ stands for the subscript
$b$ or $s$. The single- and double-layer potentials on
$\Gamma_{\beta}$ are operators that map any $\varphi\in
L^{2}(\Gamma_{\beta})$ to $\mathcal{S}_{\beta} \varphi$ and
$\mathcal{D}_{\beta} \varphi$, respectively, where
\[
\begin{aligned}\begin{aligned}\mathcal{S}_{\beta} := \mathcal{S}_{\Gamma_\beta} \mbox{ with } &  \mathcal{S}_\Gamma \varphi:=\int_{\Gamma}G(\cdot-s)\varphi(s)ds,\\
\mathcal{D}_{\beta} := \mathcal{D}_{\Gamma_\beta} \mbox{ with }  &
\mathcal{D}_{\Gamma}\varphi
 :=\int_{\Gamma}\frac{\partial
G}{\partial\nu_{s}}(\cdot-s)\varphi(s)ds,
\end{aligned}
\end{aligned}
\]
 where $G$ is the Green function for the Laplacian in
 $\mathbb{R}^{2}$:
 \begin{equation} \label{defG}
 G(x) := \frac{1}{2\pi} \log |x|, \quad x \neq 0.
 \end{equation}
For $\varphi\in L^{2}(\Gamma_{\beta})$, the functions
$\mathcal{S}_{\beta}\varphi$ and $\mathcal{D}_{\beta}\varphi$ are
harmonic functions in $\mathbb{R}^{2}\setminus\Gamma_{\beta}$;
their singularities hold on $\Gamma_{\beta}$. To describe these
singularities, we define, for a function $w$ defined in
$\mathbb{R}^{2}\setminus\Gamma_{\beta}$ and $x\in\Gamma_{\beta}$:
\[
\begin{aligned}\left.w(x)\right|_{\pm} & :=\lim_{t\rightarrow0}w(x\pm t\nu(x)),\\
\left.\frac{\partial w}{\partial\nu}(x)\right|_{\pm} & :=\lim_{t\rightarrow0}
 \nabla w(x\pm t\nu(x)) \cdot \nu(x) .
\end{aligned}
\]
 Across the boundary~$\Gamma_{\beta}$, the following  trace relations
hold~\cite{ammari2007polarization}:
\begin{equation}
\begin{aligned}\left.\mathcal{S}_{\beta}\varphi\right|_{+} & =\left.\mathcal{S}_{\beta}\varphi\right|_{-},\\
\left.\frac{\partial\mathcal{S}_{\beta}\varphi}{\partial\nu}\right|_{\pm} & =\left(\pm\frac{1}{2}I+\mathcal{K}_{\beta}^{*}\right)\varphi,\\
\left.\mathcal{D}_{\beta}\varphi\right|_{\pm} & =\left(\mp\frac{1}{2}I+\mathcal{K}_{\beta}\right)\varphi,\\
\left.\frac{\partial\mathcal{D}_{\beta}\varphi}{\partial\nu}\right|_{+} & =\left.\frac{\partial\mathcal{D}_{\beta}\varphi}{\partial\nu}\right|_{-}.
\end{aligned}
\label{eq:jump_formulas}
\end{equation}
 Here, the operator $\mathcal{K}_{\beta}$ and its $L^{2}$-adjoint
$\mathcal{K}_{\beta}^{*}$ are given by
\[
\begin{aligned}(\mathcal{K}_{\beta}\varphi)(x) & :=\frac{1}{2\pi}{\rm p.v.} \int_{\Gamma_{\beta}}\frac{
(s-x)\cdot \nu(s)}{\left|x-s\right|^{2}}\varphi(s)ds & ,\,\, x\in\Gamma_{\beta},\\
(\mathcal{K}_{\beta}^{*}\varphi)(x) & :=\frac{1}{2\pi }{\rm p.v.}
\int_{\Gamma_{\beta}}\frac{ (x-s)\cdot \nu(x)
}{\left|x-s\right|^{2}}\varphi(s)ds & ,\,\, x\in\Gamma_{\beta},
\end{aligned}
\]
 where ${\rm p.v.}$ stands for
the Cauchy principal value. From (\ref{eq:jump_formulas}) it
follows that the following jump formulas hold:
$$
\left.\frac{\partial\mathcal{S}_{\beta}\varphi}{\partial\nu}\right|_{+}
-
\left.\frac{\partial\mathcal{S}_{\beta}\varphi}{\partial\nu}\right|_{-}
= \varphi \quad \mbox{and} \quad
\left.\mathcal{D}_{\beta}\varphi\right|_{+} -
\left.\mathcal{D}_{\beta}\varphi\right|_{-} = - \varphi.
$$
The following invertibility result is useful
\cite{escauriaza1992,verchota1984}.
\begin{thm} \label{thm:invertibility_lambda_K}Suppose that
$\Gamma_{\beta}$ has Lipschitz regularity. Then the operator
$\lambda I-\mathcal{K}_{\beta}^{*}$ is invertible on
$L_{0}^{2}(\Gamma_{\beta}):=\{ \varphi \in L^2(\Gamma_{\beta}) :
\int_{\Gamma_\beta} \varphi =0\}$ if
$\left|\lambda\right|\geq1/2$, and for $\lambda\in ( -\infty,
1/2]\cup (1/2,+\infty)$, $\lambda I-\mathcal{K}_{\beta}^{*}$ is
invertible on $L^{2}(\Gamma_{\beta})$.
\end{thm} With these essentials tools, we can now prove the following
decomposition formula in the same spirit as in \cite{kang1996}:
\begin{lem} \label{lemlatter} The solution of
problem~(\ref{eq:governing_equation}) can be written as
\begin{equation}
u(x)=H(x)+(\mathcal{S}_{s}\tilde{\varphi_{s}})(x)+(\mathcal{S}_{b}\varphi_{b})(x),\label{eq:decomposition_formula}
\end{equation}
 where
\begin{equation}
H(x)=\sum_{j=1}^{m}\alpha_{j}G(x-z_{j}),\label{eq:definition_H}
\end{equation}
 and the pair $(\tilde{\varphi_{s}},\varphi_{b})\in L^{2}(\Gamma_{s})\times L^{2}(\Gamma_{b})$
is uniquely determined by
\begin{equation}
\left\{ \begin{alignedat}{2}(\lambda_{s}I-\mathcal{K}_{s}^{*})\tilde{\varphi_{s}}-\frac{\partial\mathcal{S}_{b}\varphi_{b}}{\partial\nu} & =\frac{\partial H}{\partial\nu}, & \,\, x\in\Gamma_{s},\\
(\lambda_{b}I - \mathcal{K}_{b}^{*})\varphi_{b} -
\frac{\partial\mathcal{S}_{s}\tilde{\varphi_{s}}}{\partial\nu} & =
\frac{\partial H}{\partial\nu}, & \,\, x\in\Gamma_{b}.
\end{alignedat}
\right.\label{eq:definition_phis_phib}
\end{equation}
Here, $\lambda_b$ and $\lambda_s$ are given by (\ref{deflambdas}).
Moreover, the decomposition (\ref{eq:decomposition_formula}) of
$u$ into a source part $H$ and a refraction part
$\mathcal{S}_{s}\tilde{\varphi_{s}}+\mathcal{S}_{b}\varphi_{b}$ is
unique.\end{lem}
\begin{proof} The
system~(\ref{eq:governing_equation}) is equivalent to the
following transmission problem~\cite{allaire2007numerical}:
\[
\left\{ \begin{alignedat}{2}\Delta u & ={f}, & \,\, x\in\mathbb{R}^{2}\setminus(\Gamma_{b}\cup\Gamma_{s}),\\
\left.u\right|_{+}-\left.u\right|_{-} & =0, & \,\, x\in\Gamma_{b}\cup\Gamma_{s},\\
k_{s}\left.\frac{\partial u}{\partial\nu}\right|_{+}-k_{b}\left.\frac{\partial u}{\partial\nu}\right|_{-} & =0,
& \,\, x\in\Gamma_{b},\\
\left.\frac{\partial u}{\partial\nu}\right|_{+}-k_{s}\left.\frac{\partial u}{\partial\nu}\right|_{-} & =0, &
\,\, x\in\Gamma_{s},\\
\left|u\right| & = {O}(\left|x\right|^{-1}), &
\,\,\left|x\right|\rightarrow\infty,\text{ uniformly in }\hat{x}.
\end{alignedat}
\right.
\]
The existence of a solution $(\tilde{\varphi}_s, \varphi_b)$ to
(\ref{eq:definition_phis_phib}) comes from the fact that
$|\lambda_s|, |\lambda_b|  \in (1/2, + \infty)$  and Theorem
\ref{thm:invertibility_lambda_K}. On the other hand, the functions
$\mathcal{S}_{s}\tilde{\varphi_{s}}$ and
$\mathcal{S}_{b}\varphi_{b}$ are harmonic in $\Omega_{b}$, and
according to the definition of $H$, we have $\Delta u= {f}$ in
$\Omega_{b}$. In $\Omega_{s}$ and
$\mathbb{R}^{2}\setminus\overline{\Omega}_{s}\cup
\overline{\Omega}_{b}$, all these functions are harmonic so we
have $\Delta u=0$. The trace relations on $\Gamma_b$ and
$\Gamma_s$ are then given by the singularities
(\ref{eq:jump_formulas}) of $\mathcal{S}_{s}$ and
$\mathcal{S}_{b}$ (see~\cite{ammari2007polarization}) since $H$ is
smooth away from the points~$z_{j}$. Finally, all these functions
are controlled by~$\left|x\right|^{-1}$
when~$\left|x\right|\rightarrow\infty$. In this way, the existence
of a solution to~(\ref{eq:governing_equation}) is proved.

To prove the uniqueness of the decomposition, let us take
$\tilde{\varphi_{s}}'$ and $\varphi_{b}'$ such that
\[
H+\mathcal{S}_{s}\tilde{\varphi_{s}}+\mathcal{S}_{b}\varphi_{b}=H+\mathcal{S}_{s}\tilde{\varphi_{s}}'+\mathcal{S}_{b}\varphi_{b}'.
\]
 Then, $\mathcal{S}_{s}(\tilde{\varphi_{s}}-\tilde{\varphi_{s}}')=\mathcal{S}_{b}(\varphi_{b}'-\varphi_{b})$
is harmonic in $\Omega_{s}\cup\overline{\Omega}_{b}$, which gives
by the jump formula $\varphi_{b}=\varphi_{b}'$. Finally, applying
once more the jump formula, we have
$\tilde{\varphi_{s}}=\tilde{\varphi_{s}}'$.
\end{proof}

We now check the far field condition stated in
(\ref{eq:governing_equation_rad}). Recall that $\mathcal{K}_b(1) =
\mathcal{K}_s(1)=1/2$. From
$$
\int_{\Gamma_b}
\frac{\partial\mathcal{S}_{b}\varphi_{b}}{\partial\nu} =
\int_{\Gamma_b} \varphi_b, \quad \int_{\Gamma_s}
\frac{\partial\mathcal{S}_{s}\tilde{\varphi_{s}}}{\partial\nu} =
0, \quad \int_{\Gamma_s} \frac{\partial H}{\partial \nu} =
\int_{\Gamma_b} \frac{\partial H}{\partial \nu} = \sum_j \alpha_j,
$$
by taking the average of the two equations in
(\ref{eq:definition_phis_phib}) on $\Gamma_s$ and $\Gamma_b$,
respectively, we find that
$$
\int_{\Gamma_s} \tilde{\varphi_{s}} =  (\sum_j \alpha_j)
\frac{(\lambda_b + 1/2)}{(\lambda_s -1/2) (\lambda_b -1/2)} \quad
\mbox{ and } \quad \int_{\Gamma_b} {\varphi_{b}} = \frac{\sum_j
\alpha_j}{\lambda_b -1/2},
$$
and therefore, from the representation formula
(\ref{eq:decomposition_formula}) it follows that
$$\left|u_{\delta,k_{b}} - (\sum_j \alpha_j) \frac{(\lambda_b + 1/2) (\lambda_s +
1/2)}{2 \pi (\lambda_b -1/2) (\lambda_s - 1/2)} \log |x| \right| =
{O}(\left|x\right|^{-1}), \quad
\,\,\left|x\right|\rightarrow\infty \text{ uniformly in
}\hat{x}.$$ Note that in the limit $\delta\rightarrow 0$ and $k_b
\rightarrow \infty$, the far field condition above and
(\ref{deflambdas}) yield $\lambda_b \rightarrow -1/2, \lambda_s
\rightarrow -1/2$, and therefore,
\begin{equation} \label{radiationoinft}
\left|u_{0,\infty} \right| = {O}(\left|x\right|^{-1}), \quad
\,\,\left|x\right|\rightarrow\infty \text{ uniformly in }\hat{x}.
\end{equation}
For the system~(\ref{eq:asymptotic_equation}), Lemma
\ref{lemlatter} yields the following result.
\begin{lem} \label{lem:decomposition_lemma_asymptotic} Assume that (\ref{neutre}) holds. The solution
of problem~(\ref{eq:asymptotic_equation}) can be written as
\begin{equation}
u(x)=H(x)-\frac{1}{\xi}(\mathcal{S}_{b}\varphi)(x)+(\mathcal{D}_{b}\varphi)(x),
\label{eq:decomposition_formula_asymptotic}
\end{equation}
 where $H$ is given by (\ref{eq:definition_H}) and $\varphi\in L_0^{2}(\Gamma_{b}):= \{\phi \in L^2(\Gamma_b): \int_{\Gamma_b}
 \phi =0\}$
is given by the following integral equation:
\begin{equation} \label{eq:decomposition_formula_asymptotic2}
\frac{1}{\xi}\left(\frac{1}{2}I-\mathcal{K}_{b}^{*}\right)\varphi
+\frac{\partial\mathcal{D}_{b}\varphi}{\partial\nu}=-\frac{\partial
H}{\partial\nu}, \quad x \in \Gamma_b.
\end{equation}
 The decomposition (\ref{eq:decomposition_formula_asymptotic}) of $u$ into a source part and a refraction part is unique.
\end{lem}

In the general case, (\ref{eq:decomposition_formula_asymptotic2})
should be replaced with
$$
\frac{1}{\xi}\left(\frac{1}{2}I-\mathcal{K}_{b}^{*}\right)\varphi
+\frac{\partial\mathcal{D}_{b}\varphi}{\partial\nu}=-\frac{\partial
H}{\partial\nu} + \frac{1}{|\Gamma_b|} \sum_j \alpha_j, \quad x
\in \Gamma_b.
$$
Note that since $\int_{\Gamma_b} \frac{\partial H}{\partial \nu} =
\sum_j \alpha_j$, the far field condition (\ref{radiationoinft})
is satisfied in the general case.

The proof of this lemma involves exactly the same arguments as in
the previous one: jump formulas applied to the operators.

The decomposition formulas~(\ref{eq:decomposition_formula}) and
(\ref{eq:decomposition_formula_asymptotic}) will be essential in
the next part to show that, at the first-order, $u_{\delta,k_{b}}$
converges to $u_{0,\infty}$.

\subsubsection{Asymptotic expansion of the electric potential for highly resistive
skin and highly conductive body}

\label{sub:BC-derivation}

In this part, we will use the decomposition formula for
$u_{\delta,k_{b}}$ and compute asymptotic expansions of the
refraction part. The limiting solution will then be
$u_{0,\infty}$. This latter is well defined if the limits
$\delta\rightarrow0$ and $k_{b}\rightarrow\infty$ are independent,
so we must seek the two following limits:
\[
\lim_{k_{b}\rightarrow\infty}\lim_{\delta\rightarrow0}u_{\delta,k_{b}}\mbox{ and }\lim_{\delta\rightarrow0}\lim_{k_{b}\rightarrow\infty}u_{\delta,k_{b}},
\]
 and show that they are the same. Zribi \cite[chapter 3]{zribilayer}
studied the case when $k_{b}$ remains fixed, with non-uniform
thickness of the skin $\Omega_{s}$; the limit $u_{0,1}$ is the
solution of the system:
\begin{equation}
\left\{ \begin{alignedat}{2}\Delta u_{0,1} & ={f}, & \,\, x\in\Omega_{b},\\
\Delta u_{0,1} & =0, & \,\, x\in\mathbb{R}^{2}\setminus\overline{\Omega}_{b},\\
\left.u_{0,1}\right|_{+}-\left.u_{0,1}\right|_{-} & =-\xi\left.\frac{\partial u_{0,1}}{\partial\nu}\right|_{+}, & \,\, x\in\partial\Omega_{b},\\
\left.\frac{\partial u_{0,1}}{\partial\nu}\right|_{+}-k_{b}\left.\frac{\partial u_{0,1}}{\partial\nu}\right|_{-} & =0, & \,\, x\in\Omega_{b},\\
\left|u_{0,1}\right| & = {O}(\left|x\right|^{-1}), &
\,\,\left|x\right|\rightarrow\infty,\text{ uniformly in }\hat{x}.
\end{alignedat}
\right.\label{eq:asymptotic_equation_zribi}
\end{equation}
 Here, we will follow the same outline for the proof: first we will
remind the asymptotic expansions of the operators involved in
(\ref{eq:definition_phis_phib}), and then we will match the
asymptotic expansions for $\tilde{\varphi}_{s}$ and $\varphi_{b}$.

\subsubsection*{Asymptotic expansions of the operators}

In the decomposition formula (\ref{eq:decomposition_formula}), $H$
is independent of $\delta$ and $k_{b}$; we just have to analyze
the dependence of $\tilde{\varphi}_{s}$ and $\varphi_{b}$. Remark
that from (\ref{eq:definition_phis_phib})
\begin{itemize}
\item the dependence on $k_{b}$ is carried only by $\lambda_{b}$
since $\mathcal{S}_{b}$ and $\mathcal{K}_{b}^{*}$ depend only on
the shape of $\Omega_{b}$; \item the dependence on $\delta$ is
carried by $\lambda_{s}$, $\mathcal{S}_{s}$, $\mathcal{K}_{s}^{*}$
and $\partial/\partial\nu(x)$ for $x\in\Gamma_{s}$.
\end{itemize}
In this subsection, we will focus on the asymptotic expansions of
the operators (the limits of $\lambda_{s}$ and $\lambda_{b}$ are
obvious). They have been performed in \cite{ammari2010conductivity,zribilayer};
in order to apply this proof, we first need some assumptions.

Suppose $\Gamma_{b}$ is defined in the following way:
\[
\Gamma_{b}:=g\left(\partial B\right),
\]
 where $g$ is a $\mathcal{C}^{3,\eta}$ diffeomorphism of the unit sphere $\partial B:=\partial B(0,1)$
for some $\eta>0$. Moreover, we suppose that the function
$X_g:[0,2\pi]\rightarrow\mathbb{R}^{2}$ defined by
\[
X_g =g\left(\left(\begin{array}{c}
\cos t\\
\sin t
\end{array}\right)\right),
\]
 is such that $\left|X_g'(t)\right|=1$ for all $t\in[0,2\pi]$. Thus,
$X_g$ is a $\mathcal{C}^{2,\eta}$ arclength counterclockwise
parametrization of $\Gamma_{b}$. Then the outward unit normal to
$\Omega_{b}$, $\nu(x)$ at $x=X_g(t)$, is given by
\[
\nu(x)=R_{-\frac{\pi}{2}}X_g'(t),
\]
 where $R_{-\frac{\pi}{2}}$ is the rotation by $-\pi/2$. The tangential
vector $T(x)$ at $x =X_g(t)$ is defined by
\[
T(x)=X_g'(t),
\]
 and $X_g'(t)\bot X_g''(t)$. The curvature $\tau(x)$ at $x=X_g(t)$ is defined
by
\[
X_g''(t)=\tau(x)\nu(x).
\]
 Let $\Psi_{\delta}$ be the diffeomorphism from $\Gamma_{b}$ onto
$\Gamma_{s}$ given by
\begin{equation} \label{defpsidelta}
\Psi_{\delta}(x)=x+\delta\nu(x).\end{equation}

With these assumptions, the following regularity result holds \cite{crisoforis2004}:
\begin{thm} \label{thm:analycity}Let $\eta>0$. Let, for
a Lipschitz function $g \in \mathcal{C}^{0,1}\left(\partial
B,\mathbb{R}^{2}\right)$,
\[
l_{\partial B}[g]:=\inf_{x\neq y\in\partial
B}\left|\frac{g(x)-g(y)}{x-y}\right|.
\]
Introduce the set $\mathcal{A}_{\partial B}$ of admissible
diffeomorphisms of the unit sphere:
\[
\mathcal{A}_{\partial B}:=\left\{ g\in \mathcal{C}^{1}(\partial
B,\mathbb{R}^{2}),l_{\text{\ensuremath{\partial}B}}[g]>0\right\} .
\]
 Then, for any integer $m>0$, the operators $S$ and $D$ defined on
  $\left(\mathcal{C}^{m,\eta}(\partial B,\mathbb{R}^{2})\cap\mathcal{A}_{\partial B}\right)
 \times \mathcal{C}^{m-1,\eta}(\partial B)$
($\left(\mathcal{C}^{m,\eta}(\partial
B,\mathbb{R}^{2})\cap\mathcal{A}_{\partial B}\right)\times
\mathcal{C}^{m,\eta}(\partial B)$, respectively) to
$\mathcal{C}^{m,\eta}(\partial B)$ by
\[
\begin{alignedat}{2}S[g,\varphi](x) & := & \mathcal{S}_{g(\partial B)}(\varphi \circ g^{-1})\circ g(x), & \, x\in\partial B,\\
D[g,\varphi](x) & := & \mathcal{D}_{g(\partial B)}(\varphi \circ
g^{-1})\circ g(x), & \, x\in\partial B,
\end{alignedat}
\]
 are jointly analytic with respect to their variables $g$ and $\varphi$.
\end{thm} Moreover, we have explicit formulas for the derivatives
with respect to the variable $g$ \cite{crisoforis2004}.

Then, we have the following asymptotic expansions \cite{ammari2010conductivity,zribilayer}:
\begin{prop} \label{pro:DL_operators}Let $\varphi\in\mathcal{C}^{1,\eta}(\Gamma_{b})$
and $\tilde{\psi}\in\mathcal{C}^{1,\eta}(\Gamma_{s})$ for some $\eta>0$.
Then, we have the following asymptotic expansions for $x\in\Gamma_{b}$:

\begin{equation}
\begin{aligned}\left(\mathcal{K}_{s}^{*}\tilde{\psi}\right)\circ\Psi_{\delta}(x) & =\mathcal{K}_{b}^{*}\psi(x)+\delta\mathcal{K}_{b}^{(1)}\psi(x)
+O(\delta^{2}),\\
\frac{\partial\mathcal{S}_{b}\varphi}{\partial\nu}\circ\Psi_{\delta}(x) & =\left(\frac{1}{2}I+\mathcal{K}_{b}^{*}\right)\varphi(x)+\delta\mathcal{R}_{b}\varphi(x)+O(\delta^{1+\eta}),\\
\frac{\partial\mathcal{S}_{s}\tilde{\psi}}{\partial\nu}(x) & =\left(-\frac{1}{2}I+\mathcal{K}_{b}^{*}\right)\psi(x)+\delta\mathcal{L}_{b}\psi(x)+O(\delta^{1+\eta}),
\end{aligned}
\label{eq:DL_operators}
\end{equation}
 where $\psi:=\tilde{\psi}\circ\Psi_{\delta}$, $\Psi_\delta$ being defined by (\ref{defpsidelta}), and
\begin{equation}
\begin{aligned}\mathcal{K}_{b}^{(1)}\psi(x) & =\tau(x)\mathcal{K}_{b}^{*}\psi(x)-\mathcal{K}_{b}^{*}(\tau\psi)(x)-\frac{d^{2}\mathcal{S}_{b}\psi}{dt^{2}}(x)+\frac{\partial\mathcal{D}_{b}\psi}{\partial\nu}(x),\\
\mathcal{R}_{b}\varphi(x) & =\tau(x)\left(\frac{1}{2}I+\mathcal{K}_{b}^{*}\right)\varphi(x)
-\frac{d^{2}\mathcal{S}_{b}\varphi}{dt^{2}}(x),\\
\mathcal{L}_{b}\psi(x) & =\left(\frac{1}{2}I-\mathcal{K}_{b}^{*}\right)(\tau\psi)(x)+\frac{\partial\mathcal{D}_{b}\psi}{\partial\nu}(x),
\end{aligned}
\label{eq:definition_K1_R_L}
\end{equation}
 where $d/dt$ is the tangential derivative in the direction of $T(x)=X_g' \circ X_g^{-1}(x)$.
\end{prop} Note that, according to Theorem~\ref{thm:analycity}, the constants in the $O(\delta^{1+\eta})$ terms depend on
$\left\Vert g\right\Vert _{\mathcal{C}^{3,\eta}}$.

Moreover, since the thickness of $\Omega_{s}$ is uniform, we have
$\nu\circ\Psi_{\delta}(x)=\nu(x)$, and a Taylor expansion of $H$
gives, for $x\in\Gamma_{b}$:
\begin{equation}
\frac{\partial H}{\partial\nu}\circ\Psi_{\delta}(x)=\frac{\partial
H}{\partial\nu}(x) +\delta \nu(x) \cdot \big[ D^{2}H(x)\nu(x)
\big] +O(\delta^{2}),\label{eq:DL_H}
\end{equation}
 where $D^{2}H$ denotes the Hessian of $H$.

\subsubsection*{Asymptotic expansions on the layers}

In order to prove Theorem~\ref{thm:main-result}, we will first
show the convergence on the layers (see next lemma). Then, in the
next subsection, we will extend the domain of validity by
application of the maximum principle.

The following lemma holds. \begin{lem} There exist constants $C$
and $C'$ independent of $\delta$ and $k_{b}$ such that the
following inequalities hold for $\delta$ and $1/k_{b}$ small
enough:
\begin{equation}
\begin{alignedat}{2}\left\Vert u_{\delta,k_{b}}-u_{0,\infty}\right\Vert _{L^{\infty}(\Gamma_{b})} & \leq & C\left(\delta+\frac{1}{k_{b}}\right),\\
\left\Vert u_{\delta,k_{b}}-u_{0,\infty}\right\Vert
_{L^{\infty}(\Gamma_{s})} & \leq & C^\prime
\left(\delta+\frac{1}{k_{b}}\right),
\end{alignedat}
\label{eq:estimates-layer}
\end{equation}
 where $u_{\delta,k_{b}}$ and $u_{0,\infty}$ are solutions of (\ref{eq:governing_equation})
and (\ref{eq:asymptotic_equation}), respectively.\end{lem}
\begin{proof} Only the first limit will be shown, the second one
being very similar. For this purpose, we must show that the limits
$\delta\rightarrow0$ and $k_{b}\rightarrow\infty$ are independent,
\emph{i.e.}, they commute. First, let us compute the limit of
$u_{\delta,k_{b}}$ when $\delta\rightarrow0$, and then the limit
$k_{b}\rightarrow\infty$ (which will be much easier). Then, we
will invert this process.

This first limit is the main problem in \cite[chapter
3]{zribilayer}, except that, in that study, $k_{b}=1$ and the
thickness of $\Omega_{s}$ is non-uniform. According to
theorem~\ref{thm:analycity}, the formulas in
\cite{crisoforis2004}, and by composition with the regular
diffeomorphism $\Psi_{\delta}$ from $\Gamma_{s}$ to $\Gamma_{b}$,
we have
\[
\left\Vert
\mathcal{S}_{s}\tilde{\varphi}_{s}-\mathcal{S}_{b}\varphi_{s}-\delta\left[\left(-\frac{1}{2}I+\mathcal{K}_{b}\right)\varphi_{s}-\mathcal{S}_{b}(\tau\varphi_{s})\right]
\right\Vert _{\mathcal{C}^{2,\eta}(\Gamma_{b})}\leq C\delta^{2},
\]
 where $\varphi_{s}:=\tilde{\varphi_{s}}\circ\Psi_{\delta}$. Hence,
with the help of the decomposition formula~(\ref{eq:decomposition_formula}),
we have the following asymptotic expansion uniformly on $\Gamma_{b}$:
\begin{equation}
u_{\delta,k_{b}}(x)=H(x)+\mathcal{S}_{b}(\varphi_{b}+\varphi_{s})(x)+\delta\left[\left(-\frac{1}{2}I+\mathcal{K}_{b}\right)\varphi_{s}(x)-\mathcal{S}_{b}(\tau\varphi_{s})\right](x)+O(\delta^{2}),\label{eq:DL_u}
\end{equation}
 We now look for expansions of the functions $\varphi_{s}$
and $\varphi_{b}$ when $\delta\rightarrow0$ that will be re-injected
in this equation. Using Proposition~\ref{pro:DL_operators} and (\ref{eq:definition_phis_phib}),
these functions are solutions of the following system:
\begin{multline}
\left\{ \begin{alignedat}{1}\frac{k_{s}}{k_{s}-1}\varphi_{s}-\left(\frac{1}{2}I+\mathcal{K}_{b}^{*}\right)(\varphi_{s}+\varphi_{b})+\delta\left[-\mathcal{K}_{b}^{(1)}\varphi_{s}-\mathcal{R}_{b}\varphi_{b}\right]+O(\delta^{1+\eta}) & =\\
\frac{\partial H}{\partial\nu}+\delta \big[\nu \cdot D^{2}H\nu \big]  +O(\delta^{2}),\\
\frac{k_{s}}{k_{s}-k_{b}}\varphi_{b}+\left(-\frac{1}{2}I+\mathcal{K}_{b}^{*}\right)(\varphi_{s}+\varphi_{b})
+\delta\mathcal{L}_{b}\varphi_{s}+O(\delta^{1+\eta}) &
=-\frac{\partial H}{\partial\nu}.
\end{alignedat}
\right.\label{eq:modified_system}
\end{multline}
 Let us define the formal asymptotic expansions:
\[
\left\{ \begin{alignedat}{1}\varphi_{s} & =\frac{1}{\delta}\varphi_{s}^{(-1)}+\varphi_{s}^{(0)}+\delta\varphi_{s}^{(1)}+\ldots,\\
\varphi_{b} & =\frac{1}{\delta}\varphi_{b}^{(-1)}+\varphi_{b}^{(0)}+\delta\varphi_{b}^{(1)}+\ldots.
\end{alignedat}
\right.
\]
 Aiming to have the $0$-order term in the expansion (\ref{eq:DL_u}),
here we seek for the terms of order $-1$ and $0$. By substitution
into (\ref{eq:modified_system}) and identification of the leading-order
terms in the first line, we get:
\[
\left(\frac{1}{2}I+\mathcal{K}_{b}^{*}\right)\left(\varphi_{s}^{(-1)}+\varphi_{b}^{(-1)}\right)=0,
\]
 so that, by Theorem~\ref{thm:invertibility_lambda_K}, we have:
\begin{equation}
\varphi_{s}^{(-1)}+\varphi_{b}^{(-1)}=0.\label{eq:phis+phib_ordre-1}
\end{equation}
 Let us now look at the $0$-order terms; summing the two lines, we
get:
\[
\frac{1}{\xi}\left(\varphi_{s}^{(-1)}+\frac{1}{k_{b}}\varphi_{b}^{(-1)}\right)+(\varphi_{s}^{(0)}+\varphi_{b}^{(0)})
+
\left[\mathcal{K}_{b}^{(1)}\varphi_{s}^{(-1)}+\mathcal{R}_{b}\varphi_{b}^{(-1)}-\mathcal{L}_{b}\varphi_{s}^{(-1)}\right]=0,
\]
 which gives, with the help of (\ref{eq:definition_K1_R_L}) and (\ref{eq:phis+phib_ordre-1}),
\begin{equation}
\varphi_{s}^{(0)}+\varphi_{b}^{(0)}=\left[\left(\frac{1}{k_{b}}-1\right)
\frac{1}{\xi}
+\tau\right]\varphi_{s}^{(-1)}.\label{eq:phis0+phib0}
\end{equation}
 This quantity is what we need in (\ref{eq:DL_u}); thus, only $\varphi_{s}^{(-1)}$
remains to be found. This can be done by identification of the
$0$-order terms in the first line of (\ref{eq:modified_system})
and using the definitions of $\mathcal{K}_{b}^{(1)}$ and
$\mathcal{R}_{b}$ given by (\ref{eq:definition_K1_R_L}):
\begin{equation}
\frac{1}{\xi}\varphi_{s}^{(-1)}+\frac{1}{\xi}\left(\frac{1}{k_{b}}-1\right)\left(\frac{1}{2}I+\mathcal{K}_{b}^{*}\right)\varphi_{s}^{(-1)}+\frac{\partial\mathcal{D}_{b}\varphi_{s}^{(-1)}}{\partial\nu}=-\frac{\partial
H}{\partial\nu}.\label{eq:definition_phis_-1}
\end{equation}
 Finally, the expansion (\ref{eq:DL_u}) yields:
\begin{equation}
u_{\delta,k_{b}}(x)=H(x)+\left[\frac{1}{\xi}\left(\frac{1}{k_{b}}-1\right)\mathcal{S}_{s}+\left(-\frac{1}{2}I+\mathcal{K}_{b}\right)\right]\varphi_{s}^{(-1)}(x)+O(\delta).\label{eq:u_delta_sigma_ordre1}
\end{equation}
 This leading-order term (denoted $u_{0,k_{b}}$) verifies (\ref{eq:asymptotic_equation_zribi})
according to (\ref{eq:definition_phis_-1}) and jump formulas~(\ref{eq:jump_formulas}).

The asymptotic $k_{b}\rightarrow\infty$ does not add further
difficulty. Indeed, let us define the following asymptotic:
\[
\varphi_{s}^{(-1)}=\varphi_{s}^{(-1,0)}+\frac{1}{k_{b}}\varphi_{s}^{(-1,1)}+\ldots.
\]
 By substitution into equation (\ref{eq:definition_phis_-1}) and
identification of the leading-order terms, we get:
\[
\frac{1}{\xi}\left(\frac{1}{2}I-\mathcal{K}_{b}^{*}\right)\varphi_{s}^{(-1,0)}+\frac{\partial\mathcal{D}_{b}\varphi_{s}^{(-1,0)}}{\partial\nu}=-\frac{\partial H}{\partial\nu},
\]
 and then the expansion~(\ref{eq:u_delta_sigma_ordre1}) becomes:
\begin{equation}
u_{\delta,k_{b}}(x)=H(x)-\frac{1}{\xi}\mathcal{S}_{s}\varphi_{s}^{(-1,0)}(x)+\left(-\frac{1}{2}I+\mathcal{K}_{b}\right)\varphi_{s}^{(-1,0)}(x)+O\left(\delta\right),\label{eq:DL_u_final}
\end{equation}
 which is (\ref{eq:decomposition_formula_asymptotic}) applied on
$\Gamma_{b}$ according to the jump formula of
$\mathcal{D}_{b}$~(\ref{eq:jump_formulas}). Hence, according to
lemma~\ref{lem:decomposition_lemma_asymptotic}, the first-order
asymptotic of $u_{\delta,k_{b}}$ is $u_{0,\infty}$.

Let us now show that the limits $\delta\rightarrow0$ and
$k_{b}\rightarrow\infty$ commute: unlike in the previous
subsection, we will first perform the limit
$k_{b}\rightarrow\infty$ and then the limit $\delta\rightarrow0$.
Given the fact that
\[
\lambda_{b}=-\frac{1}{2}+O\left(\frac{1}{k_{b}}\right),
\]
 the definition of $\tilde{\varphi_{s}}$ and $\varphi_{b}$ in (\ref{eq:definition_phis_phib})
will be affected only in the second line. Indeed, with the following
expansions:
\[
\left\{ \begin{alignedat}{1}\tilde{\varphi_{s}} & =\tilde{\varphi}_{s}^{(0)}+\frac{1}{k_{b}}\tilde{\varphi}_{s}^{(1)}+\ldots,\\
\varphi_{b} & =\varphi_{b}^{(0)}+\frac{1}{k_{b}}\varphi_{b}^{(1)}+\ldots,
\end{alignedat}
\right.
\]
 this second line becomes, at the leading order:
\[
\left(-\frac{1}{2}I+\mathcal{K}_{b}^{*}\right)(\varphi_{s}^{(0)}+\varphi_{b}^{(0)})+\delta\mathcal{L}_{b}\varphi_{s}^{(0)}+O(\delta^{1+\eta})=-\frac{\partial H}{\partial\nu},
\]
 where $\varphi_{s}^{(0)}:=\tilde{\varphi_{s}}^{(0)}\circ\Psi_{\delta}$.
With the expansion:
\[
\left\{ \begin{alignedat}{1}\varphi_{s}^{(0)} & =\frac{1}{\delta}\varphi_{s}^{(0,-1)}+\varphi_{s}^{(0,0)}+\delta\varphi_{s}^{(0,1)}+\ldots,\\
\varphi_{b}^{(0)} & =\frac{1}{\delta}\varphi_{b}^{(0,-1)}+\varphi_{b}^{(0,0)}+\delta\varphi_{b}^{(0,1)}+\ldots,
\end{alignedat}
\right.
\]
 the identifications (\ref{eq:phis+phib_ordre-1}), (\ref{eq:phis0+phib0}),
and (\ref{eq:definition_phis_-1}) respectively become:
\begin{equation}
\varphi_{s}^{(0,-1)}+\varphi_{b}^{(0,-1)}=0,\label{eq:phis+phib_ordre-1(bis)}
\end{equation}

\begin{equation}
\varphi_{s}^{(0,0)}+\varphi_{b}^{(0,0)}=\left[\tau-\frac{1}{\xi}\right]\varphi_{s}^{(0,-1)},\label{eq:phis0+phib0(bis)}
\end{equation}

\begin{equation}
\frac{1}{\xi}\left(\frac{1}{2}I-\mathcal{K}_{b}^{*}\right)\varphi_{s}^{(0,-1)}+\frac{\partial\mathcal{D}_{b}\varphi_{s}^{(0,-1)}}{\partial\nu}=-\frac{\partial H}{\partial\nu}.\label{eq:definition_phis_-1(bis)}
\end{equation}
 Finally, recalling that the expansion of $u_{\delta,k_{b}}$ in (\ref{eq:DL_u})
is conductivity-independent, we obtain the same expansion (\ref{eq:DL_u_final}).
\end{proof}

\subsubsection*{Proof of Theorem~\ref{thm:main-result}}

With the estimates (\ref{eq:estimates-layer}) on the layers $\Gamma_{b}$
and $\Gamma_{s}$, we are now ready to prove the estimate (\ref{eq:estimate-general})
on the whole space applying the maximum principle.

For the sets $\Omega_{b}$ and $\Omega_{s}$, it is straightforward:
the function $u_{\delta,k_{b}}-u_{0,\infty}$ is harmonic in these
bounded domains, so the maximum is reached on the boundaries~\cite{taylor1}.
And, since this maximum is dominated by $\delta$ and $1/k_{b}$,
we have:
\[
\left\Vert u_{\delta,k_{b}}-u_{0,\infty}\right\Vert _{L^{\infty}(\bar{\Omega_{b}}
\cup\Omega_{s})}\leq C\left(\delta+\frac{1}{k_{b}}\right).
\]
For the exterior domain, we cannot apply directly the maximum
principle since this domain is unbounded. However, the conditions
at infinity in the systems (\ref{eq:governing_equation}) and
(\ref{eq:asymptotic_equation}) allow us to have a similar control.
Indeed, this condition tells us that
\begin{equation}
\left\Vert u_{\delta,k_{b}}-u_{0,\infty}\right\Vert
_{L^{\infty}(B(0,R))}=O(R^{-1}).\label{eq:infinite_condition_error}
\end{equation}
 We take:
\[
\varepsilon:=\frac{1}{2}\left\Vert u_{\delta,k_{b}}-u_{0,\infty}\right\Vert _{L^{\infty}(\bar{\Omega_{b}}\cup\Omega_{s})},
\]
 and choose $R_{0}$ such that, for $R\geq R_{0}$, the right-hand side of
(\ref{eq:infinite_condition_error}) is bounded by $\varepsilon$.
Then, we have:
\[
\left\Vert u_{\delta,k_{b}}-u_{0,\infty}\right\Vert
_{L^{\infty}(\mathbb{R}^{2}\setminus B(0,R_{0}))}\leq\varepsilon.
\]
 Now, only the bounded domain $B(0,R_{0})\setminus\left(\overline{\Omega}_{b}\cup\Omega_{s}\right)$
remains, where we can apply the maximum principle. Thus, Theorem~\ref{thm:main-result}
is proved.

\subsection{Final formulation and notation}

In the previous subsections, we have performed a multi-scale
analysis of the problem to identify the effective equations with
boundary conditions. In order to make things clear, let us
summarize the results and simplify the notation.

The electric potential emitted by the fish is the solution of the
complex-conductivity equation (\ref{eq:EQS-final}) with boundary
conditions given by the system (\ref{eq:asymptotic_equation}). It
is easy to see that in the case of an inhomogeneity outside the
body, these boundary conditions will not be changed because the
asymptotics are done with the layer potentials of the domains
defining the fish.

Hence, we conclude this section by summing up the results:
omitting all the subscripts, the electric potential $u$ is the
solution of the system
\begin{equation}
\left\{ \begin{alignedat}{2}\Delta u & = {f}, & \,
\, x\in \Omega,\\
\nabla\cdot (1+ (k -1 +i\varepsilon\omega) \chi_D) \nabla u & = 0,
& \,
\, x\in\mathbb{R}^{2}\setminus\overline{\Omega},\\
{} u \big|_+ - u \big|_- -\xi\left.\frac{\partial u}{\partial\nu}\right|_{+} & =0, & \,\, x\in\Gamma,\\
\left.\frac{\partial u}{\partial\nu}\right|_{-} & =0, & \,\, x\in\Gamma,\\
\left|u\right| & = {O}(\left|x\right|^{-1}), &
\,\,\left|x\right|\rightarrow\infty,\text{ uniformly in }\hat{x},
\end{alignedat}
\right.\label{eq:direct_problem-final}
\end{equation}
where $\chi_D$ is the characteristic function of the target $D$,
$k+ i \varepsilon \omega$ is the conductivity inside $D$, $\omega$
is the frequency, and $k$ and $\varepsilon$ are positive
constants. Here, we have assumed that $\sum_j \alpha_j=0$. In the
case where it is not, we should replace the boundary condition
$\frac{\partial u}{\partial\nu}\big|_{-} =0$ on $\Gamma$ with
$\frac{\partial u}{\partial\nu}\big|_{-} = ({1}/{|\Gamma|}) \sum_j
\alpha_j$. From now on, we restrict ourselves to the case $\sum_j
\alpha_j =0$. Note that taking two points $z_1$ and $z_2 \in
\Omega$ close enough and $\alpha_1 = - \alpha_2 \neq 0$ yields an
approximation of a dipole at $(z_1+z_2)/2$ of moment $|\alpha_1|$
and direction orthogonal to $(z_1-z_2)$.

\section{Detection algorithm for multi-frequency measurements}

\label{sec:detection_algo}

In this section, we develop an algorithm to recover (from a single
measurement) the location of a small object located far away from
the fish. This algorithm is based on multi-frequency measurements,
as it is explained in subsection~\ref{sub:multifreq}. In
subsection~\ref{sub:dipolar-expansion}, asymptotic expansions will
be carried out for the electric field in the presence of a small
and distant target. Finally, the algorithm will be explained in
detail in subsection~\ref{sub:algorithm}.

\subsection{Multi-frequency measurements}

\label{sub:multifreq}

Let us suppose that the electric current produced by the electric
organ, (\emph{i.e.}, the source term $f$ in equation
(\ref{eq:direct_problem-final})) is periodically time-dependent
with separation of variables, that is
\[
f(x,t)={f}(x)h(t),
\]
 where ${f}$ is a sum of Dirac functions and $h(t)$
is periodic with fundamental frequency $\omega_{0}$. Hence, we set
\begin{equation}
h(t)=\sum_{n=1}^N h_{n}e^{in\omega_{0}t},\label{eq:formule-h}
\end{equation}
 where $N$ is the upper bound that ensures the low-frequency regime.
According to the previous section, the electric potential $u$ is
then given by
\begin{equation} \label{sumu}
u(x,t)=\sum_{n=1}^N h_{n}u_{n}e^{in\omega_{0}t},
\end{equation}
 where $u_{n}$, for $n=1,\ldots, N$,  is solution of the following system
\begin{equation}
\left\{ \begin{alignedat}{2} \Delta u_n & = {f}, & \,
\, x\in \Omega,\\
\nabla\cdot (1+ (k -1 +i\varepsilon n \omega_0) \chi_D) \nabla u_n
& = 0, & \,
\, x\in\mathbb{R}^{2}\setminus\overline{\Omega},\\
{} u_{n} \big|_+ - u_{n}  \big|_- - \xi\left.\frac{\partial u_{n}}{\partial\nu}\right|_{+} & =0, & \,\, x\in\Gamma,\\
\left.\frac{\partial u_{n}}{\partial\nu}\right|_{-} & =0, & \,\, x\in\Gamma,\\
\left|u_{n}\right| & = {O}(\left|x\right|^{-1}), &
\,\,\left|x\right|\rightarrow\infty,\text{ uniformly in }\hat{x},
\end{alignedat}
\right.\label{eq:complex_neumann_problem}
\end{equation}

\subsection{A dipolar expansion in the presence of a target}

\label{sub:dipolar-expansion}

In this subsection, we derive useful formulas in order to simplify
the data. For the sake of simplicity, and for numerical reasons
that will be given in section \ref{sec:numeric}, only one target
$D$ will be considered.

The electroreceptors of the fish measure the electric current at
the surface of the skin \cite{moller1995}. Hence, from a single
measurement, we can construct the Space-Frequency Response (SFR)
matrix $A$, whose terms are given by
\[
A_{ln}=\left(\left.\frac{\partial
u_{n}}{\partial\nu}\right|_{+}-\left.\frac{\partial
U}{\partial\nu}\right|_{+}\right)(x_{l}),\,\,\textrm{ for } 1\leq
n\leq N\textrm{ and }1\leq l\leq L,
\]
 where $\left(x_{l}\right)_{1\leq l\leq L}$ are points on the boundary
$\Gamma$ and $U$ is the static background solution, \emph{i.e.},
the electric potential without any target which does not depend on
$n$. It is the solution of (\ref{eq:complex_neumann_problem}) with
a constant conductivity equal to $1$ outside the body $\Omega$.

The first formula, given in Proposition \ref{propos2}, is often
called a \emph{dipolar expansion}; indeed, in the presence of a
small inhomogeneity, the perturbation of the electric potential
looks like the electric potential of a dipole
\cite{ammari2004reconstruction, cedio1998identification}. More
precisely, using exactly the same arguments as in \cite[Chapter
4]{ammari2004reconstruction} and in \cite{ammarisima02} we have
the following result.

\begin{prop} \label{propos2}

If $D:=z+\alpha B$ with ${\rm dist}(z,\Gamma)\gg1$, $\alpha\ll1$
and $B$ is an open set, then we have
\begin{equation}
A_{ln}\simeq-\alpha^{2}\nabla U(z)^T M(k_{n},B)
\nabla_{z}\left(\left.\frac{\partial
G_{R}}{\partial\nu_{x}}\right|_{+}\right)(x_{l},z),\label{eq:SFR-first-approx}
\end{equation}
 where $T$ denotes the transpose, $k_{n}= k + i \varepsilon \omega_0 n$ is the (complex) conductivity of the
 target at the
frequency $n\omega_{0}$, $M(k_n,B) = (M_{\alpha
\beta}(k_n,B))_{\alpha, \beta=1,2}$ is the first-order
polarization tensor associated to $B$ with conductivity $k_n$
\cite{ammari2007polarization}:
$$
M_{\alpha \beta}(k_n,B):= \int_{\partial B} (\lambda_n I -
\mathcal{K}_{B}^{*})^{-1}(\nu_\alpha) y_\beta \, ds(y), \quad
\lambda_n:= \frac{k_n +1}{2(k_n -1)}, \quad \alpha, \beta=1,2,
$$
 and $G_{R}$ is the Green function
associated to Robin boundary conditions, which is defined for
$z\in \mathbb{R}^{2}\setminus\overline{\Omega}$ by
\begin{equation}
\left\{ \begin{alignedat}{2}-\Delta_{x}G_{R}(x,z) & =\delta_{z}(x), & \,\, x\in\mathbb{R}^{2}\setminus\overline{\Omega},\\
{}\left. G_{R} \right|_{+} -  \xi\left.\frac{\partial G_{R}}{\partial\nu_{x}}\right|_{+} & =0, & \,\, x\in\Gamma,\\
\left|G_{R} + \frac{1}{2\pi} \log |x| \right| & =
{O}(\left|x\right|^{-1}), &
\,\,\left|x\right|\rightarrow\infty,\text{ uniformly in }\hat{x}.
\end{alignedat}
\right.\label{eq:green-fonction-robin}
\end{equation}
\end{prop}
\begin{proof} Let
$$
H_n = - \mathcal{S}_\Gamma(\frac{\partial u_n}{\partial
\nu}\big|_+) + \mathcal{D}_\Gamma (u_n \big|_+).
$$
We have
$$
u_n - U = - (k_n -1) \int_D \nabla u_n \cdot \nabla G_R,
$$
and on the other hand,
$$
u_n - H_n = - (k_n -1) \int_D \nabla u_n \cdot \nabla G.
$$
From the transmission condition
$$
\frac{\partial u_n}{\partial \nu}\big|_+ -  k_n \frac{\partial
u_n}{\partial \nu}\big|_- = 0 \quad \mbox{on } \partial D,
$$
it follows that
\begin{equation} \label{scaling}
u_n - U = \int_{\partial D} (\lambda_n I - \mathcal{K}_{\partial
D}^{*})^{-1}(\frac{\partial H_n}{\partial \nu}) G_R.
\end{equation}
Since
$$
\| \nabla H_n - \nabla U\|_{L^\infty(D)} \leq C \alpha^2,
$$
for some constant $C$, provided that ${\rm dist}(D, \partial
\Omega)\gg \alpha$, a scaling of the integral in (\ref{scaling})
together with a Taylor expansion of $G_R$ gives the desired
asymptotic expansion. Note that the approximation in
(\ref{eq:SFR-first-approx}) is uniform in $l$ and $k_n$
\cite{ammari2007polarization}.
\end{proof}

Now, we will carry on a second formula in order to simplify this
equation. Indeed, the Green function associated to Robin boundary
conditions is tedious to compute. Instead, we will post-process
the data thanks to the following lemma which generalizes Lemma
2.15 in \cite{ammari2004reconstruction}.

\begin{lem} \label{lemgreen}
Let  $G$ denote the Green function in the free space defined by
(\ref{defG}). For $z\in\mathbb{R}^{2}\setminus \overline{\Omega}$
and $x\in\Gamma$, let $
$$G_{z}(x)=G(x-z)$ and $G_{R,z}(x)=G_{R}(x-z)$. Then
\[
\left(\frac{1}{2}I-\mathcal{K}_{\Gamma}^{*}
-\xi\frac{\partial\mathcal{D}_{\Gamma}}{\partial\nu}\right)\left(\frac{\partial
G_{R,z}}{\partial\nu_{x}}\right)(x)= -\frac{\partial G_{z}
}{\partial\nu_{x}}(x).
\]
\end{lem}
\begin{proof} Employing the same argument as in Lemma \ref {lem:decomposition_lemma_asymptotic} yields
$$
G_{R,z}= - G_z + \frac{1}{\xi}(\mathcal{S}_{\Gamma}\varphi) +
\mathcal{D}_{\Gamma}\varphi  - \mathcal{S}_{\Gamma}
(\frac{\partial G_{z} }{\partial\nu}) - \xi \mathcal{D}_\Gamma
(\frac{\partial G_{z} }{\partial\nu}) ,$$ where $\varphi= \xi
\frac{\partial G_{R,z}}{\partial \nu} \big|_+$. Therefore, taking
the normal derivative of the above identity and using the trace
relations (\ref{eq:jump_formulas}) give the result.
\end{proof}

Hence, after a calculation of $\left.\frac{\partial
u_{n}}{\partial\nu}\right|_{+}-\left.\frac{\partial
U}{\partial\nu}\right|_{+}$ on $\Gamma$, we will apply the
post-processing operator given in Lemma \ref{lemgreen}. The
modified matrix will still be denoted $A$.

To conclude, the location of the target $D$ is going to be
recovered from the knowledge of the following data
\begin{equation} \label{defdata}
\begin{alignedat}{1}A_{ln} & =\left(\frac{1}{2}I-\mathcal{K}_{\Gamma}^{*}-\xi\frac{
\partial\mathcal{D}_{\Gamma}}{\partial\nu}\right)\left(\left.\frac{\partial u_{n}}{\partial\nu}\right|_{+}-\left.
\frac{\partial U}{\partial\nu}\right|_{+}\right)(x_{l}), \quad 1\leq l\leq L, 1 \leq n\leq N, \\
\end{alignedat}
\end{equation}
which is approximately equal to
\begin{equation}
\begin{alignedat}{1}A_{ln}
 & \simeq \alpha^{2}\nabla U(z)^T M(k_{n},B) \nabla_{z}\left(\left.\frac{\partial G}{\partial\nu_{x}}\right|_{+}
 \right)(x_{l},z),
\end{alignedat}
\label{eq:SFR-final}
\end{equation}
when the characteristic size of the target $\alpha$  is small. It
is worth mentioning that the  polarization tensor $M(k_{n},B)$ is
symmetric (but not Hermitian) \cite{ammari2007polarization}.

\subsection{A location search algorithm}

\label{sub:algorithm}

Scholz described in \cite{scholz2002towards} a way to recover the
location of a target from multi-frequency measurements. The paper
focuses on an application in electrical impedance tomography (EIT)
for breast cancer detection; the algorithm was called
``Space-Frequency MUSIC''. Indeed, it is based on the so-called
MUSIC algorithm, which is a standard tool in signal theory for the
identification of several signals with an additive noise
\cite{schmidt1986multiple, devaney2004super}. It has then been
applied to identify small conductivity inhomogeneities in
\cite{ammarihanke, ammarinumer, bruhl2003direct}. In this section,
we apply a similar approach for our model.

As we can see in formula (\ref{eq:SFR-final}), the rows of the SFR
matrix are - to leading-order - linear combinations of the
derivatives of $\partial G/\partial\nu_{x}$. Moreover, one has to
distinguish whether the target is a disk or not. Indeed, in
dimension $2$ and in the case of an ellipse whose semi-axes are on
the $x_{i}$-axis and of length $a$ and $b$, the polarization
tensor $M(k,B)$, for $k\in \mathbb{C}$,  takes the form
\cite{milton_book}
\[
M(k,B)=(k-1)\vert B\vert\left(\begin{array}{cc}
\frac{a+b}{a+kb} & 0\\
0 & \frac{a+b}{b+ka}
\end{array}\right).
\]
Hence, the polarization tensor is proportional to the identity
matrix if and only if $a=b$, \emph{i.e.}, $B$ is a disk; this
result remains true in dimension $3$
\cite{ammari2007polarization}. This changes dramatically the range
of $A$: if $B$ is a disk, the response matrix
has rank $1$ and if it is an ellipse, it has rank $2$.\\
For the sake of simplicity, let us suppose that $B$ is the unit disk.
The identification process will be based on the following fact

\begin{lem}\label{lem:one-to-one}The following map
\[\begin{array}{ll}
\Lambda: &\mathbb{R}^{2}\setminus\overline{\Omega} \rightarrow  L^{2}(\Gamma)\\
\nm &z \mapsto  \nabla U(z)^T \nabla_{z}\frac{\partial
G}{\partial\nu_{x}}(\cdot,z),
\end{array}
\]
is one-to-one.\end{lem}

\begin{proof}

Suppose that $z$ and $z'$ are points on
$\mathbb{R}^{2}\setminus\overline{\Omega}$ such that
$\Lambda(z)=\Lambda(z'):=\varphi$. Let us define the two following
functions
\[
\begin{array}{ll}v_{z}: &\mathbb{R}^{2}\setminus\overline{\Omega}\cup\{z\} \rightarrow \mathbb{R}\\
& x \mapsto  \nabla U(z)^T \nabla_{z}G(x,z),
\\
\nm \nm \nm
v_{z'}: & \mathbb{R}^{2}\setminus\overline{\Omega}\cup\{z'\} \rightarrow  \mathbb{R}\\
& x \mapsto  \nabla U(z')^T \nabla_{z'}G(x,z').
\end{array}
\]
Thus, these two functions both solve the following boundary value
problem
\[
\left\{ \begin{alignedat}{2}\Delta v & =0, & \,\, x\in\mathbb{R}^{2}\setminus\overline{\Omega}\cup\{z\}\cup\{z'\},\\
\frac{\partial v}{\partial\nu} & =\varphi, & \,\, x\in\Gamma,\\
v & \rightarrow0 & \,\,\left|x\right|\rightarrow\infty,\text{ uniformly in }\hat{x}.
\end{alignedat}
\right.
\]
Hence, by the uniqueness of the solution for this problem, we have
\[
\nabla U(z)\cdot\nabla_{z}G(x,z)=\nabla
U(z')\cdot\nabla_{z'}G(x,z'),\mbox{ for all
}x\in\mathbb{R}^{2}\setminus\overline{\Omega}\cup\{z\}\cup\{z'\}.
\]
Relying on the singularity of $G(\cdot,z)$ at the point $z$, this
is only possible if $z=z'$.

\end{proof}

However, we do not have access to the complete function (because
there is only a finite number of electroreceptors on the body),
and the formula for $\Lambda$ is only an approximation, based
on~(\ref{eq:SFR-final}). The location of the target will then be
approximated as follows. In the following we suppose for the sake
of simplicity that $x_1,\ldots, x_L$ are equi-distributed on
$\Gamma$.

\begin{prop}[Space-Frequency MUSIC]

\label{prop:SF-MUSIC} Define the vector
\begin{equation}
\tilde{g}(z):=\left(\nabla U(z)\cdot\nabla_{z}\left(\frac{\partial
G}{\partial\nu_{x}}\right)(x_{1},z),\ldots,\nabla
U(z)\cdot\nabla_{z}\left(\frac{\partial
G}{\partial\nu_{x}}\right)(x_{L},z)\right)^{*},\label{eq:illumination-vector-disk}
\end{equation}
and its normalized version $g=\tilde{g}/\vert\tilde{g}\vert$.
Then, in the limit $L\rightarrow+\infty$ and $\alpha
\rightarrow0$, the following imaging functional will have a large
peak at $z$:
\begin{equation}
\mathcal{I}(z_{s}):=\frac{1}{\left|(I-P)g(z_{s})\right|},\label{eq:imaging_functional}
\end{equation}
 where $P$ is the orthogonal projection onto the first
singular vector of the SFR matrix $A$.

\end{prop}

\begin{proof}

First of all, let us rewrite (just for this proof) the projection
$P_{\alpha}^{L}$ and the illumination vector $g^{L}$, in order to
take into account the dependence with respect to $L$. When $L$
goes to infinity, quadrature formulas show us that
\[
\vert(I-P_{\alpha}^{L})g^{L}(z_{s})\vert_{\mathbb{R}^{L}}\rightarrow\left|(I-P_{\alpha})
\frac{\Lambda(z_{s})}{\vert\Lambda(z_{s})\vert_{L^{2}(\Gamma)}}\right|_{L^{2}(\Gamma)}.
\]
Here, $P_{\alpha}$ is the projection onto the first singular
vector of the operator $\mathbb{A}_{\alpha}$, acting on the space
of functions that have the form  (\ref{eq:formule-h})
\[
\mathbb{A}_{\alpha}:h\mapsto\left(\frac{1}{2}I-\mathcal{K}_{\Gamma}^{*}
-\xi\frac{\partial\mathcal{D}_{\Gamma}}{\partial\nu}\right)\left(\left.\frac{\partial
u}{\partial\nu}\right|_{+}-\left.\frac{\partial
U}{\partial\nu}\right|_{+}\right),
\]
where $u$ is given by (\ref{sumu}) and $U$ is the background
solution ({\it i.e.}, the solution of
(\ref{eq:complex_neumann_problem}) with $\chi_D=0$).

In the limit $\alpha \rightarrow0$, $\mathbb{A}_{\alpha}$ is
approximated by the operator $\mathbb{A}:\, h\mapsto\Lambda(z)h$,
which is obviously of rank one. By theory of
perturbation~\cite{kato1976perturbation}, one has therefore
\[
\left|(I-P_{\alpha})\frac{\Lambda(z_{s})}{\vert\Lambda(z_{s})\vert_{L^{2}(\Gamma)}}
\right|_{L^{2}(\Gamma)}\rightarrow\left|(I-P)\frac{\Lambda(z_{s})}{\vert\Lambda(z_{s})
\vert_{L^{2}(\Gamma)}}\right|_{L^{2}(\Gamma)},\,\,\alpha\rightarrow0,
\]
where $P$ is the projector onto the first significant singular
vector of $\mathbb{A}$. Then, from Lemma~\ref{lem:one-to-one},
this functional is zero if and only if $z_{s}=z$.

\end{proof}

Moreover, in order to have a general algorithm which is robust
with respect to the background solution, we will plot the
following imaging functional:
\begin{equation} \label{musicf}
\mathcal{I}(z_{s}):=\max\left(\frac{1}{\left|(I-P)g^{\mathcal{E}}(z_{s})\right|},
\frac{1}{\left|(I-P)g^{\mathcal{D}}(z_{s})\right|}\right),
\end{equation}
where $g^{\mathcal{D}}$ is defined in Proposition~\ref{prop:SF-MUSIC}
and $g^{\mathcal{E}}(z_{s})$ is the normalization of the following
vector
\[
\tilde{g}^{\mathcal{E}}(z)=\left(\nabla_{z}\left(\frac{\partial
G}{\partial\nu_{x}}\right)(x_{1},z),\ldots,\nabla_{z}\left(\frac{\partial
G}{\partial\nu_{x}}\right)(x_{L},z)\right)^{*}.
\]
Numerical results will be given in section \ref{sec:numeric}.

Let us highlight the fact that in the case of a general shape $B$,
we do not know theoretically what happens. Indeed, when $k$ is
real, $M(k,B)$ is equivalent to the polarization tensor of an
ellipse \cite{bruhl2003direct}, but this is not true when
$k\in\mathbb{C}$ because the proof relies on the spectral theorem.
Here, we still have symmetry \cite{ammari2007polarization}, but it
is not sure if $M(k,B)$ is diagonalizable or not. However, we will
see in the numerical subsection \ref{sec:numeric} that the
algorithm works with shapes other than ellipses and disks.

\section{Numerical simulations}

\label{sec:numeric}

In this section, numerical results are presented in order to
illustrate the multi-frequency location search algorithm
introduced in the previous section. In the first subsection, we
explain the method used to compute the electric field; this will
be the input of our location search algorithm that will be
performed in the second subsection.

\subsection{Direct problem}

\label{sub:direct-problem-numeric}

This section is devoted to the computation of the electric field
around the fish.

\subsubsection{The case without target}

The electric field $U$ generated by the fish is the function
$u_{0,\infty}$ treated in section \ref{sec:forward_problem}. Let
us recall that it is the solution of the following system:
\begin{equation}
\left\{ \begin{alignedat}{2}\Delta U & ={f}, & \,\, x\in\Omega,\\
\Delta U & =0, & \,\, x\in\mathbb{R}^{2}\setminus\overline{\Omega},\\
U\big|_+ - U\big|_- -\xi\left.\frac{\partial U}{\partial\nu}\right|_{+} & =0, & \,\, x\in\Gamma,\\
\left.\frac{\partial U}{\partial\nu}\right|_{-} & =0, & \,\, x\in\Gamma,\\
\left|U\right| & = {O}(\left|x\right|^{-1}), &
\,\,\left|x\right|\rightarrow\infty,\text{ uniformly in }\hat{x}.
\end{alignedat}
\right.\label{eq:sytem-U-developped}
\end{equation}
Numerical simulations will be done using a boundary element method
(BEM). Indeed, we need accuracy on the skin of the fish, and the
jumps at the boundaries are too difficult to handle with a finite
element method. Moreover, it reduces the number of discretization
points, resulting in a much faster algorithm.

This BEM simulation relies on the representation formula for $U$
in terms of the layer potentials. From
Lemma~\ref{lem:decomposition_lemma_asymptotic}, we have
$U=H+\mathcal{S}_{\Gamma}\psi+\mathcal{D}_{\Gamma}\varphi,$ where
$\Delta H={f}$ in the whole space, and the potentials are
solutions of the system:
\begin{equation}
\left\{ \begin{alignedat}{1}\varphi & =-\xi\psi, \quad x \in \Gamma, \\
\left(\frac{I}{2}-\mathcal{K}_{\Gamma}^{*}+\xi\frac{\partial\mathcal{D}_{\Gamma}}
{\partial\nu}\right)\psi & =\frac{\partial H}{\partial\nu}, \quad
x \in \Gamma.
\end{alignedat}
\right.\label{eq:system_potential_U}
\end{equation}
Note that we have changed a little bit the notation, in order to
 be able to test the case $\xi=0$.
On smooth domains, the operator $\mathcal{K}_{\Gamma}^{*}$ is easy
to handle because its kernel has integrable singularity, whereas
the operator $\partial\mathcal{D}_{\Gamma}/\partial\nu$ is an
\emph{hypersingular operator}. Thus, one has to perform a
integration by parts in order to regularize it: for two smooth
functions $v_{1}$ and $v_{2}$, we have (for example from
\cite[Theorem 1]{nedelec1982integral} and \cite[Theorem
6.15]{steinbach2008numerical}):
\begin{equation}
\int_\Gamma \frac{\partial\mathcal{D}_{\Gamma} v_{1}}{\partial\nu}
\cdot v_{2}
=\int_{\Gamma}\int_{\Gamma}G(x-y)\textrm{curl}_{\Gamma}v_{1}
(x)\cdot\textrm{curl}_{\Gamma}v_{2}(y)\, ds(x)\, ds(y),
\label{eq:hypersingular-integration-by-parts}
\end{equation}
where $\textrm{curl}_{\Gamma}$ is the surface rotational, defined
in the following way in dimension $2$. First, let us define the
vector:
\[
\underline{\textrm{curl}}_{\Gamma}\tilde{v}=\left(\begin{alignedat}{1}\frac{\partial\tilde{v}}{\partial x_{2}}\\
-\frac{\partial\tilde{v}}{\partial x_{1}}
\end{alignedat}
\right),
\]
where $\tilde{v}$ is an extension of $v$ into a neighborhood of
$\Gamma$, \emph{i.e.}, $\tilde{v}(x)=v\left(\mathcal{P}(x)\right)$
with  the local projection $\mathcal{P}$ onto $\Gamma$. Then $
$$\textrm{curl}_{\Gamma}$ is defined by
\[
\textrm{curl}_{\Gamma}v(x):=\nu(x)\cdot\underline{\textrm{curl}}_{\Gamma}\tilde{v}(x).
\]

In our context, this can be made much easier. Recalling the
notation of subsection \ref{sub:BC-derivation}, we have
\[
\Gamma=\left\{ x=X(t)=\left(\begin{array}{c}
X_{1}(t)\\
X_{2}(t)
\end{array}\right),\, t\in[0,2\pi]\right\} .
\]
Thus we have, for $x\in\Gamma$,
\[
\begin{alignedat}{1}\textrm{curl}_{\Gamma}v(x) & =\nu_{1}(x)\frac{\partial\tilde{v}}{\partial x_{2}}(x)
-\nu_{2}(x)\frac{\partial\tilde{v}}{\partial x_{1}} (x)\\
 & =X'_{2}(t)\frac{\partial v}{\partial x_{2}}(X(t))+X'_{1}(t)\frac{\partial v}{\partial x_{1}}(X(t)), \quad t=X^{-1}(x),\\
 & =\frac{d}{dt}\left[v(X(t))\right].
\end{alignedat}
\]
Hence, denoting by $v'$ the curvilinear derivative of $v$ on
$\Gamma$, formula (\ref{eq:hypersingular-integration-by-parts})
becomes
\[
\int_\Gamma \frac{\partial\mathcal{D}_{\Gamma} v_{1}}{\partial\nu}
\cdot v_{2} = \int_\Gamma \mathcal{S}_{\Gamma}v'_{1} \cdot v'_{2}.
\]
This enables us to derive a BEM formulation of the system (\ref{eq:system_potential_U});
however one has to perform it with $\mathbb{P}_{1}$ elements instead
of simple $\mathbb{P}_{0}$ elements in the case of $\xi=0$.

The discretization process is classical
\cite{steinbach2008numerical}. We only precise that the equation
is penalized in order to handle the condition at infinity, which
fixes an additive constant. To conclude, let us mention that this
boundary element formulation can be extended to the
three-dimensional case (see \cite{nedelec1982integral}).

\subsubsection{The case with a target}

In this subsection, we derive the modification induced on the
system (\ref{eq:system_potential_U}) in the presence of a target
$D \Subset \mathbb{R}^{2}\setminus \overline{\Omega}$ of (complex)
conductivity $k$. The system (\ref{eq:sytem-U-developped})
becomes:
\begin{equation}
\left\{ \begin{alignedat}{2}\Delta u & ={f}, & \,\, x\in\Omega,\\
\Delta u & =0, & \,\, x\in\mathbb{R}^{2}\setminus\left(\overline{\Omega}\cup\partial D\right),\\
u\big|_+ - u \big|_- -\xi\left.\frac{\partial u}{\partial\nu}\right|_{+} & =0, & \,\, x\in\Gamma,\\
\left.\frac{\partial u}{\partial\nu}\right|_{-} & =0, & \,\, x\in\Gamma,\\
{}u\big|_+ - u \big|_-  & =0, & \,\, x\in \partial D,\\
\left.\frac{\partial u}{\partial\nu}\right|_{+}-k\left.\frac{\partial u}{\partial\nu}\right|_{-} & =0, & \,\, x\in \partial
D,\\
\left|u\right| & = {O}(\left|x\right|^{-1}), &
\,\,\left|x\right|\rightarrow\infty,\text{ uniformly in }\hat{x}.
\end{alignedat}
\right.\label{eq:system-u-developped}
\end{equation}
Thus, $u$ can be written as
\[
u(x)=H(x)+\mathcal{S}_{\Gamma}\psi(x)+\mathcal{D}_{\Gamma}\varphi(x)+\mathcal{S}_{\partial
D}\phi(x).
\]
The absence of $\mathcal{D}_{\partial D}$ is justified by the
continuity across the boundary of $D$. From the jump formulas
(\ref{eq:jump_formulas}), the conditions on the boundaries
$\Gamma$ and $\partial D$ given in (\ref{eq:system-u-developped})
leads us to the following system:
\begin{equation}
\left\{ \begin{alignedat}{1}\varphi & =-\xi\psi, \quad x \in \Gamma, \\
\left(\frac{I}{2}-\mathcal{K}_{\Gamma}^{*}+\xi\frac{\partial\mathcal{D}_{\Gamma}}{\partial\nu}\right)\psi
-\frac{\partial}{\partial\nu}\left.\left(\mathcal{S}_{\partial D}\phi\right)\right|_{\Gamma} &
=\left.\frac{\partial H}{\partial\nu}\right|_{\Gamma}, \quad x \in \Gamma, \\
-\frac{\partial}{\partial\nu}\left.\left(\mathcal{S}_{\Gamma}\psi\right)\right|_{\partial
D}-\xi\frac{\partial}{\partial\nu}\left.\left(\mathcal{D}_{\Gamma}\psi\right)\right|_{\partial
D}+\left(\lambda I-\mathcal{K}_{\partial D}^{*}\right)\phi &
=\left.\frac{\partial H}{\partial\nu}\right|_{\partial D}, \quad x
\in \partial D,
\end{alignedat}
\right.\label{eq:system-potentials-anomaly}
\end{equation}
where
\[
\lambda:=\frac{k+1}{2(k-1)}.
\]
System (\ref{eq:system-potentials-anomaly}) can be rewritten as
follows:
\[
\mathbb{M}\left(\begin{alignedat}{1}\psi\\
\phi
\end{alignedat}
\right)=\left(\begin{alignedat}{1}\left.\frac{\partial H}{\partial\nu}\right|_{\Gamma}\\
\left.\frac{\partial H}{\partial\nu}\right|_{\partial D}
\end{alignedat}
\right),
\]
with
\[
\mathbb{M}:=\left(\begin{matrix} \left( \ds
\frac{I}{2}-\mathcal{K}_{\Gamma}^{*}
+\xi\frac{\partial\mathcal{D}_{\Gamma}}{\partial\nu}\right) &\ds
\left(-\left.\frac{\partial\mathcal{S}_{\partial D}}{\partial\nu}\right|_{\Gamma}\right)\\
\nm \ds
-\left(\left.\frac{\partial\mathcal{S}_{\Gamma}}{\partial\nu}\right|_{\partial
D}+\xi\left.\frac{\partial
\mathcal{D}_{\Gamma}}{\partial\nu}\right|_{\partial D}\right) &
\ds \left(\lambda I-\mathcal{K}_{\partial D}^{*}\right)
\end{matrix}
\right).
\]
The BEM formulation is then also classical, because the only
difficulty is due to the hypersingular operator in the upper left
term. Hence, we discretize $\psi\in L_0^{2}(\Gamma)$ with
$\mathbb{P}_{1}$ elements and $\phi\in L_{0}^{2}(\partial D)$ with
$\mathbb{P}_{0}$ elements.

\subsubsection{Direct simulations}

In this subsection, we present some numerical simulations of the
direct problem. We approximate the shape of the fish by an ellipse
with semi-axes of lengths $1$ and $0.3$; the electric organ is a
dipole in the $x_1$-direction of moment $1$, placed at
$z_{0}=(0.7,0)$ and the impedance is $\xi=0.1$. A ball of infinite
conductivity (more precisely, with $\sigma=10^{10}$ and
$\varepsilon=0$) and radius $r=0.05$ is located at
$(1.5\cos(\pi/3),1.5\sin(\pi/3))$. Figure
\ref{fig:fish_like_simul} shows the isopotentials. In Figure
\ref{fig:fish_like_simul} (b) it can be seen that the
isopotentials avoid the target since it is of infinite
conductivity.


\begin{figure}
\centering%
\begin{tabular}{ccc}
\includegraphics[width=7.4cm]{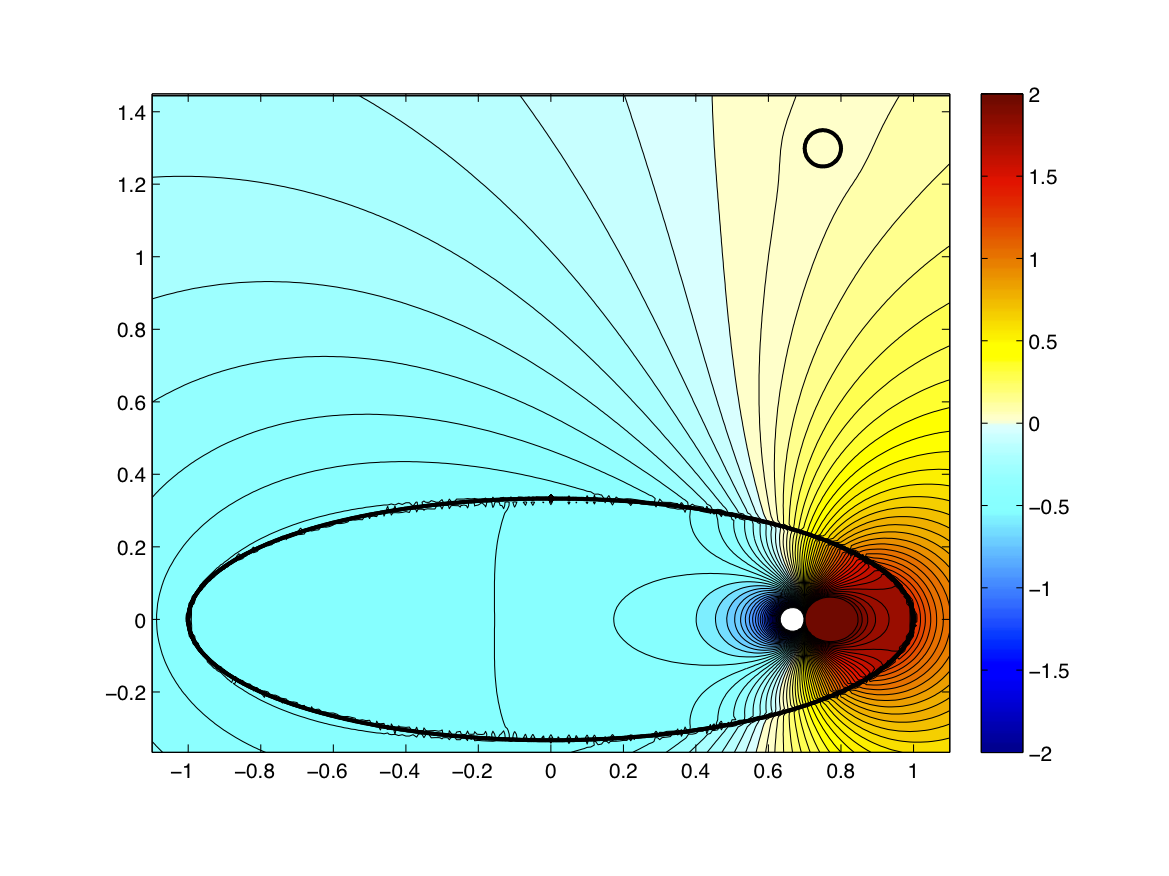} &
\includegraphics[width=7cm]{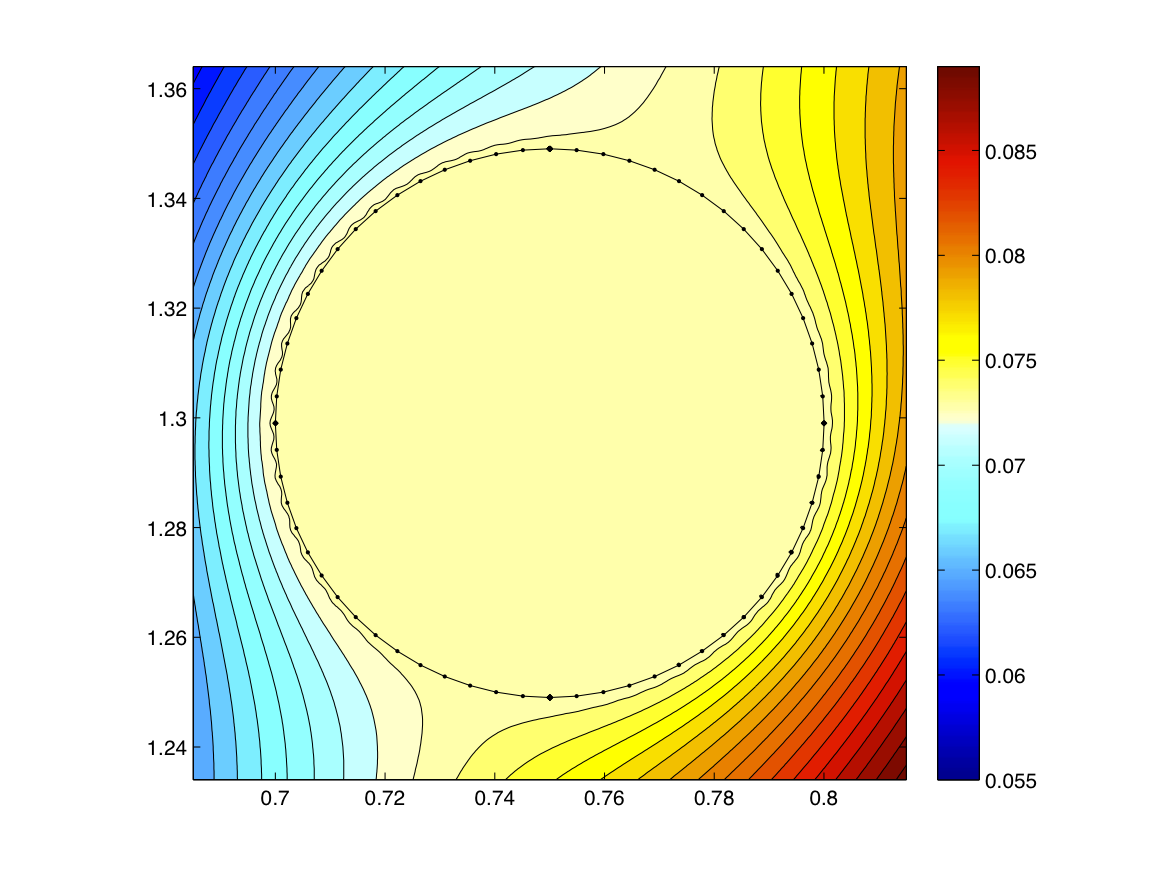}
\tabularnewline (a) & (b) \tabularnewline
\end{tabular}
\caption{\label{fig:fish_like_simul}Isopotentials of the case
described: (a) global overview (b) zoom on the target.}
\end{figure}

\subsection{Target location}

In this subsection, we show numerical target location results
using the imaging function (\ref{musicf}). With the same
parameters used for Figure \ref{fig:fish_like_simul} for the fish,
but with a small target of electric parameters $\sigma=2$ and
$\varepsilon=1$, we obtain the imaging functional plotted in
Figure \ref{fig:SF-MUSIC} (a). We use $10$ frequencies
equidistributed from $1$ to $10$. In Figure \ref{fig:SF-MUSIC}
(b) and (c), we have tested other shapes for the target.

\vspace{1cm}

\begin{figure}[!h]
\centering
\includegraphics[width=7.cm]{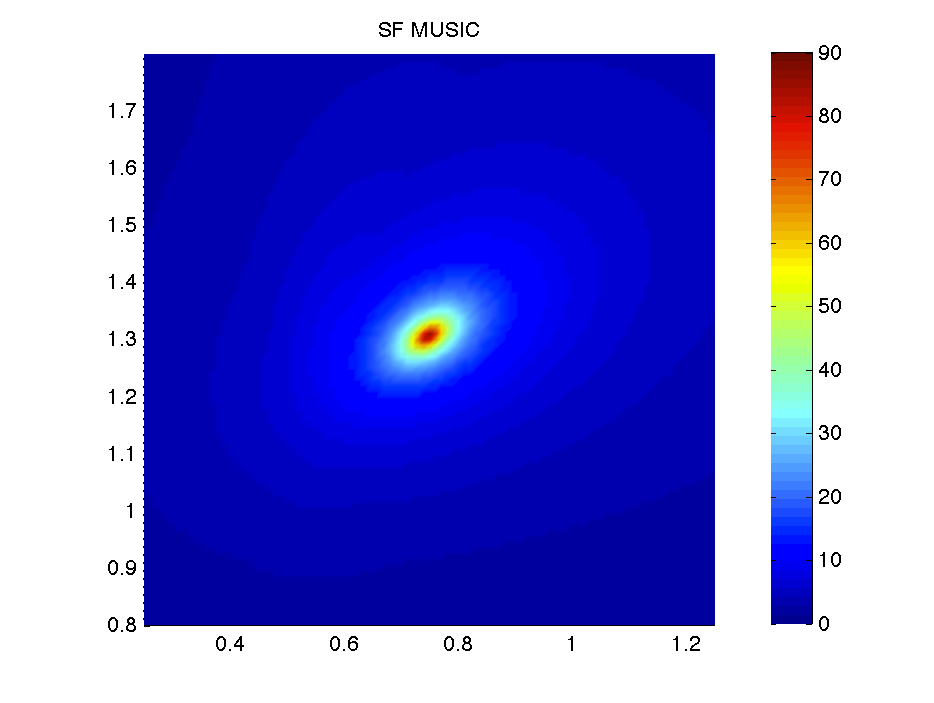} \hspace{0.5cm}
\includegraphics[width=7.cm]{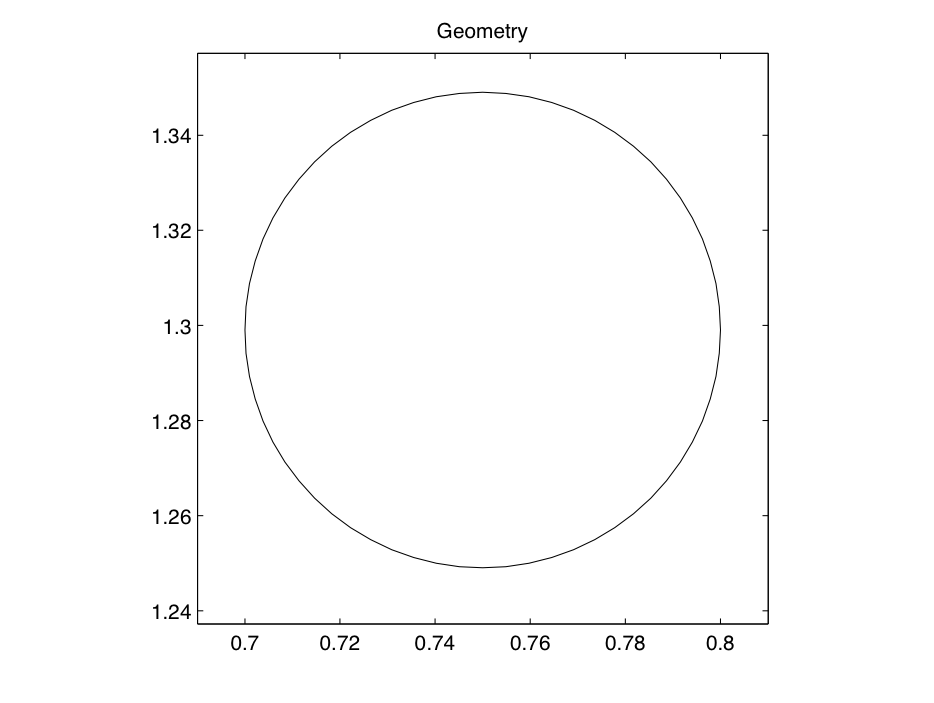}\\
\includegraphics[width=7.cm]{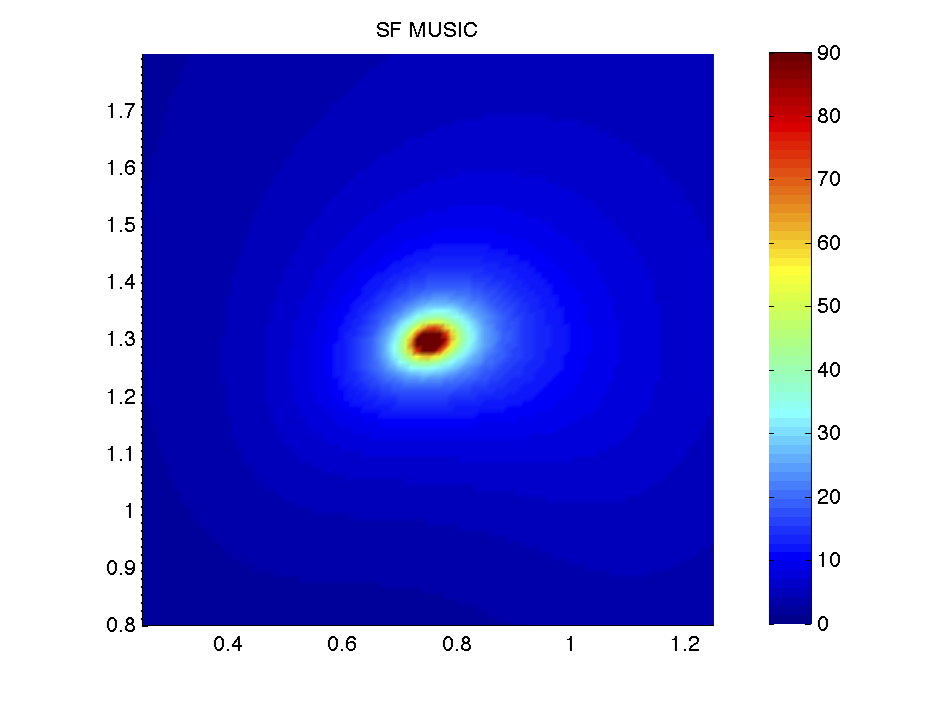}
\hspace{0.5cm} \includegraphics[width=7cm]{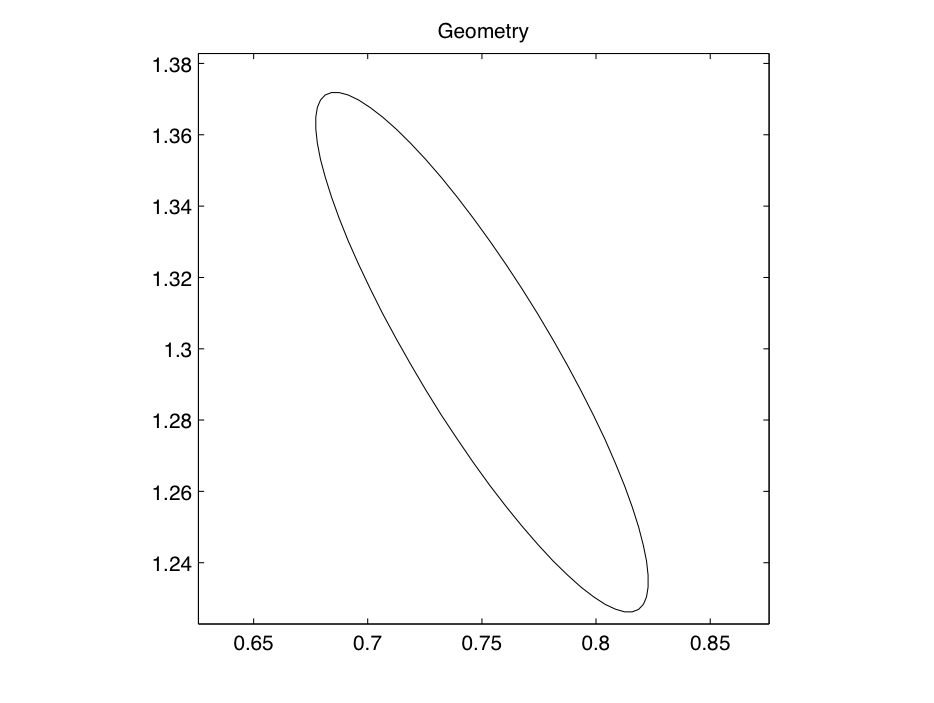} \\
\includegraphics[width=7cm]{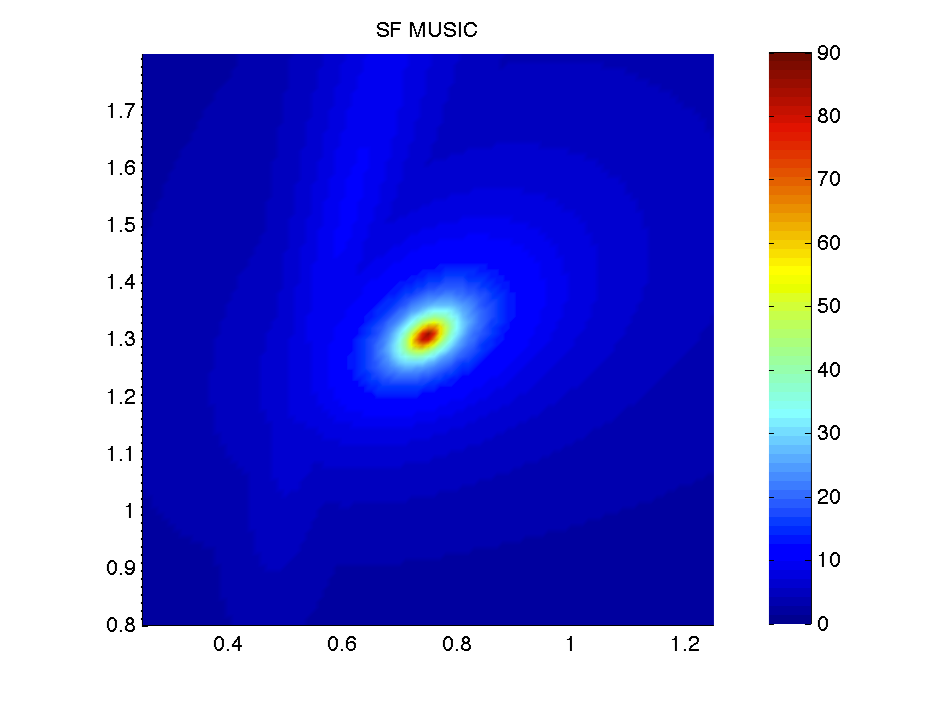}
 \hspace{0.5cm} \includegraphics[width=7cm]{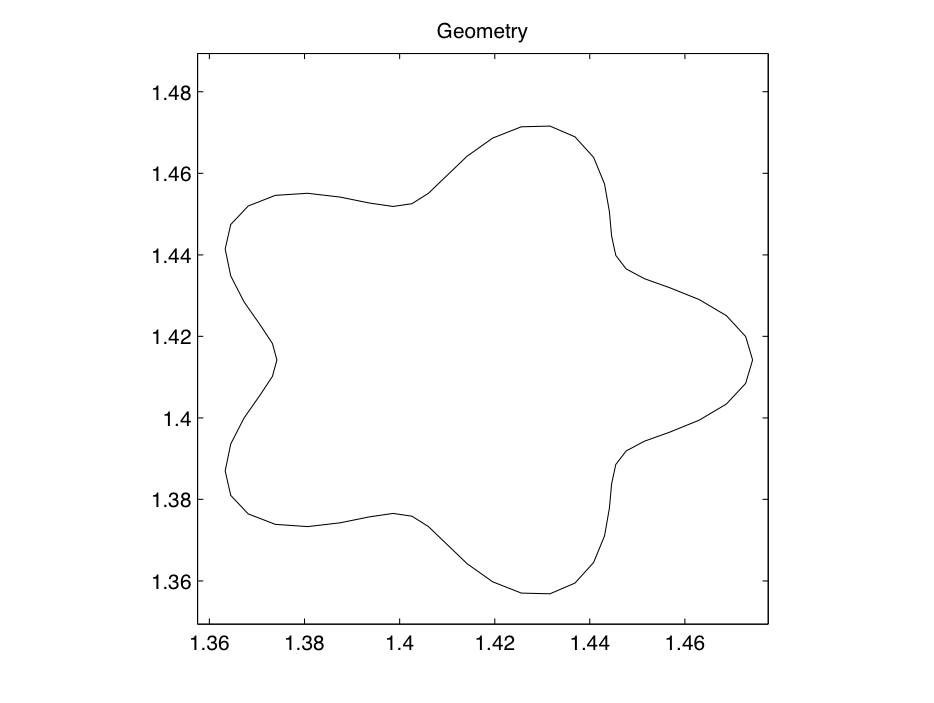}
\caption{\label{fig:SF-MUSIC}Detection (left) of the target with
the SF-MUSIC algorithm, for different target shapes (right). Here,
the number of used frequencies is $10$, equidistributed from $1$
to $10$, and there are $64$ equidistant sensors on the fish.}
\end{figure}

\subsubsection*{Stability estimates with respect to measurement noise}

Let us first notice that, in the absence of noise, the number of
used frequencies does not change  significantly the image. Indeed,
we can see in Figure \ref{fig:no-noise} that we can recover the
location of the target  with only one frequency.

\begin{figure}[!h]
\centering
\includegraphics[width=7cm]{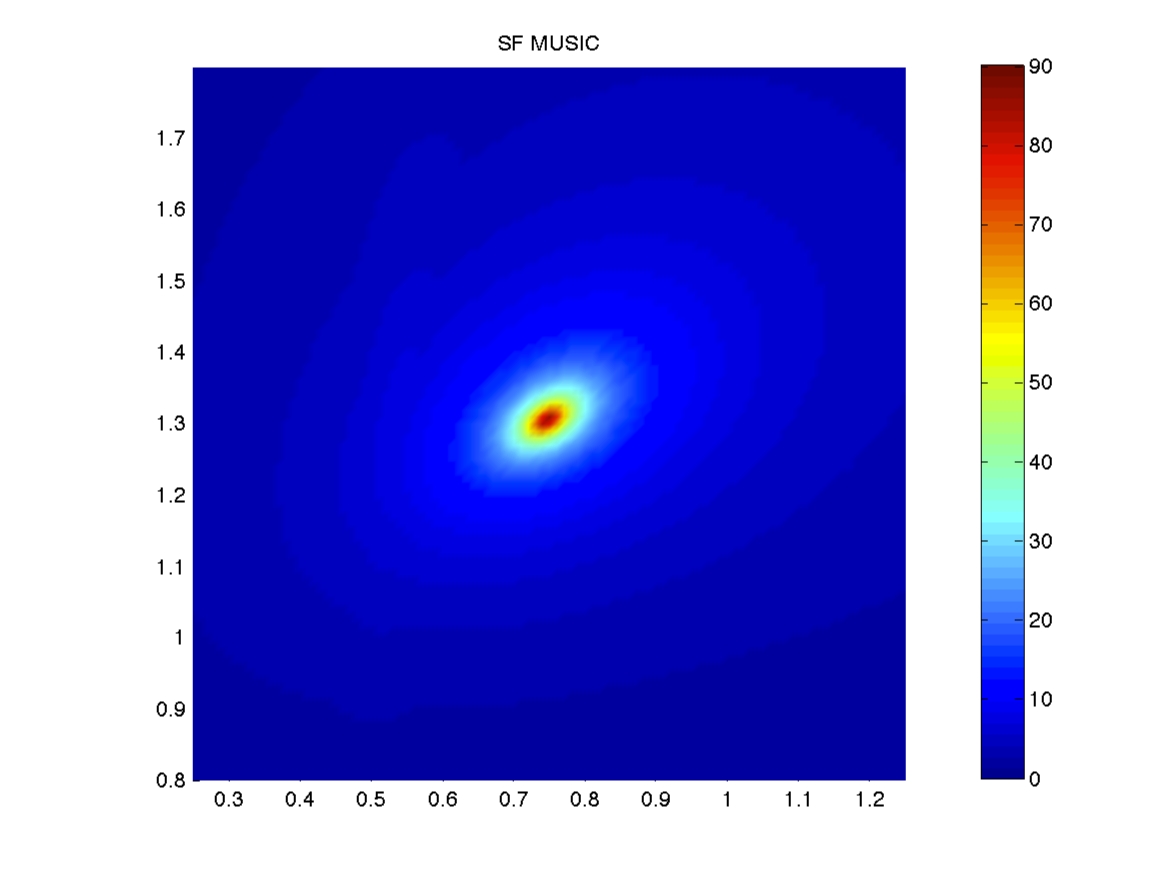}
\caption{\label{fig:no-noise}Target detection in the absence of
noise, with only one frequency equal to $1$. Here, the target is a
disk with center $(1.5\cos(\pi/3),1.5\sin(\pi/3))$ and radius
$0.05$, like in Figure \ref{fig:SF-MUSIC}(a); the number of
sensors is the same.}

\end{figure}

Let us now consider the effect of measurement noise on the
performance of the location search algorithm. We add to the
entries of the matrix $A$ defined in (\ref{eq:SFR-final})
independent Gaussian random variables of mean $0$ and standard
deviation
\[
\sqrt{\zeta}\max_{l,n}\left|\left(\left.\frac{\partial
u_{n}}{\partial\nu}\right|_{+}-\left.\frac{\partial
U}{\partial\nu}\right|_{+}\right)(x_{l})\right|.
\]
The parameter $\zeta$ is the relative strength of the noise, and
will be given in $\%$. Figure~\ref{fig:noise-freq_qualitative}
shows that increasing the number of frequencies stabilizes the
image.

\begin{figure}
\centering%
\includegraphics[width=7cm]{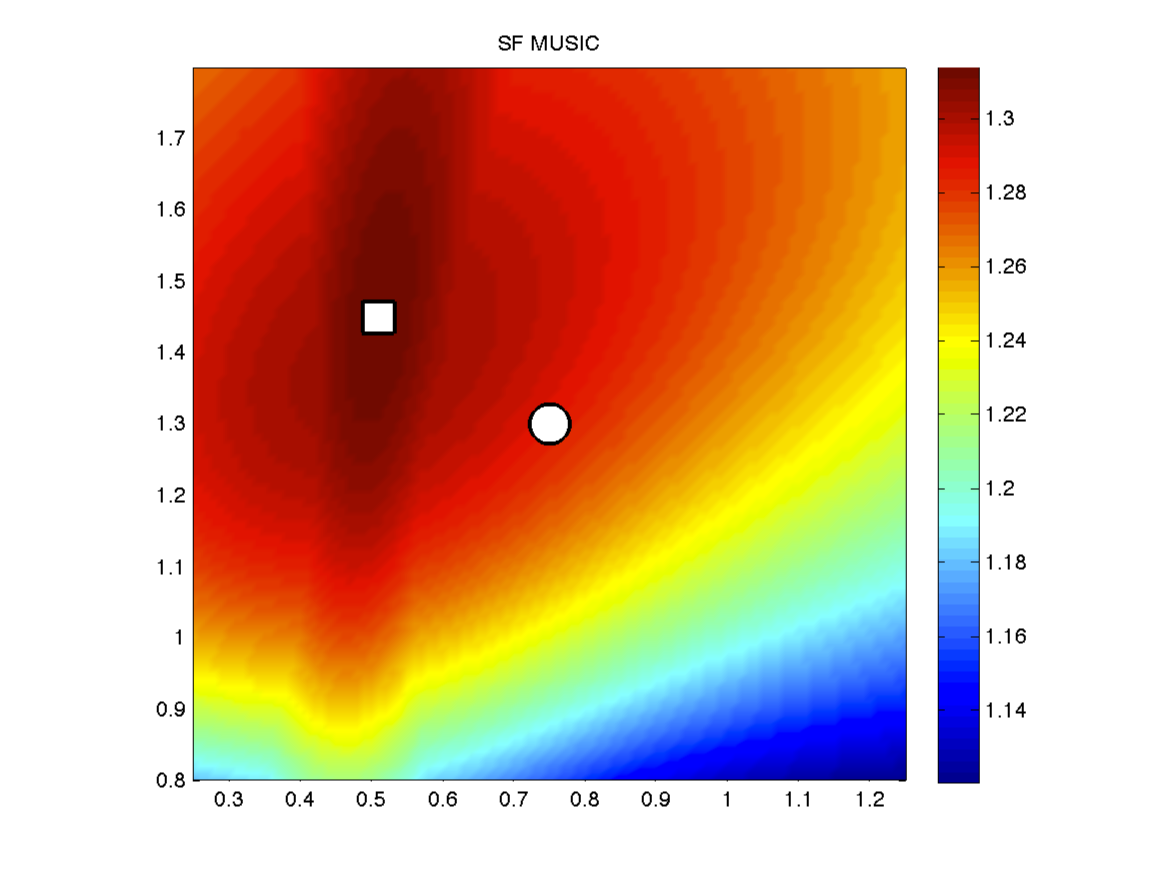} \hspace{0.2cm}
\includegraphics[width=7cm]{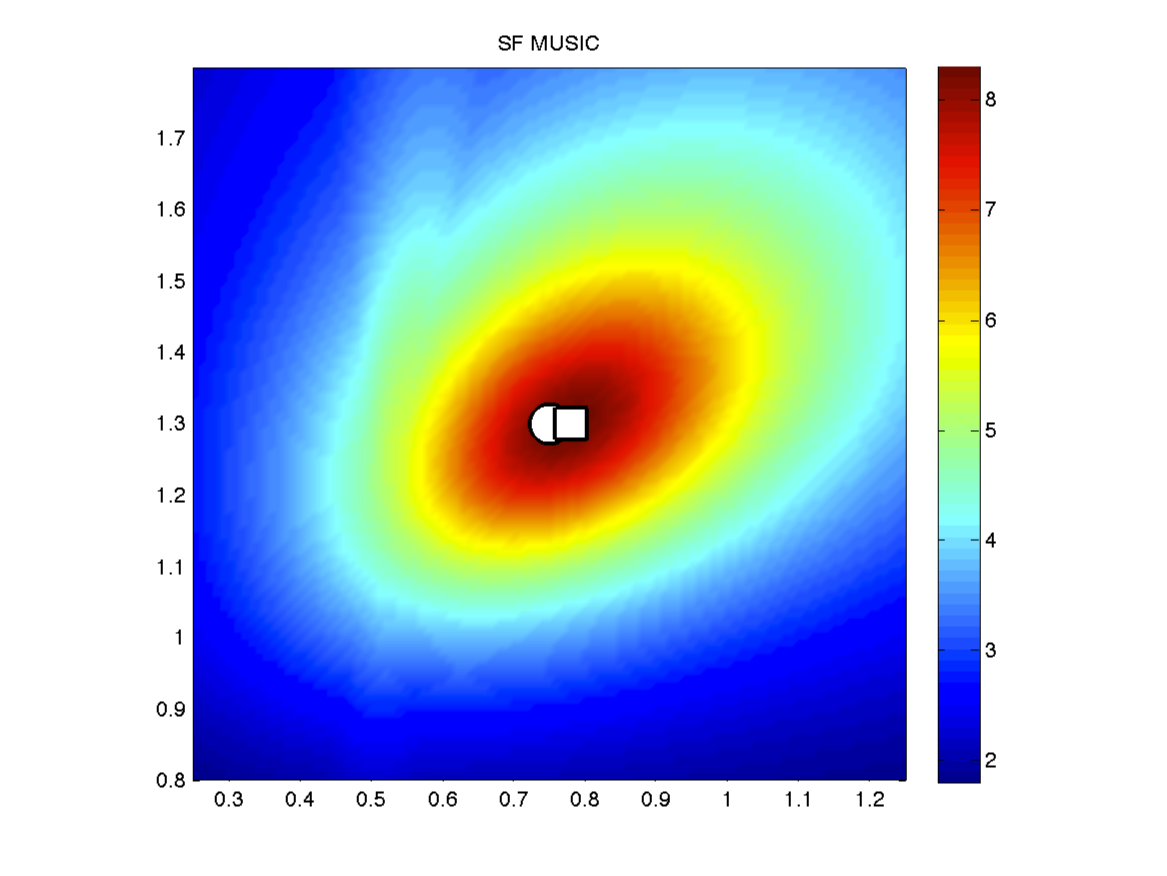}\tabularnewline
\caption{\label{fig:noise-freq_qualitative} Influence of the
number of used frequencies on the stability. Here, the same target
as in Figure \ref{fig:no-noise} is imaged with $1\%$ of noise and
~$1$~frequency (left), ~$100$~frequencies equidistributed from $1$
to $100$ (right), with $64$ sensors. The disks plot the exact
position, and the squares plot the location of the maximum of the
imaging functional.}

\end{figure}

More quantitatively, we have computed the empirical root mean
square location error (between the exact location of the target
and the maximum of the imaging functional), for $N_{r}=250$
trials. Here, the same target as in Figure \ref{fig:no-noise} is
considered. Results are shown in
Figure~\ref{fig:stats-freq-noise}.

\begin{figure}
\centering\includegraphics[height=7cm]{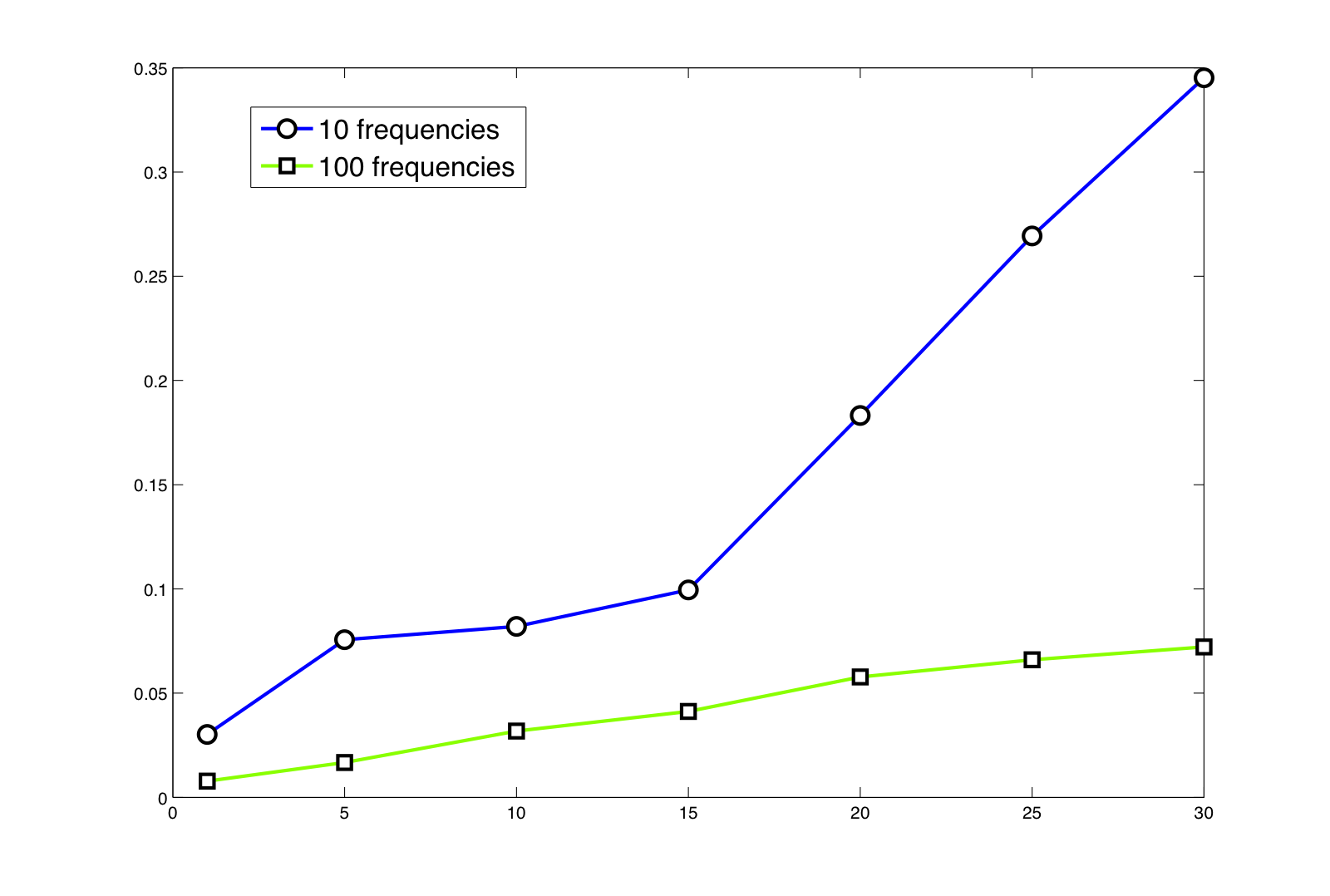}

\caption{\label{fig:stats-freq-noise}Influence of the number of
frequencies on the root mean square location error for $250$
trials. Here, the horizontal axis is for the measurement noise
level in $\%$ and the vertical axis is for the root mean square
location error.}
\end{figure}

A natural question is whether taking different values for the
frequencies plays a role. In Figure \ref{fig:noise-freq_100}, we
use the data obtained by $100$ trials for $1\%$ of noise, $64$
sensors, and a single frequency equal to $1$. Figure
\ref{fig:noise-freq_100} shows that the values of the frequencies
do not play a crucial role in the location procedure. In fact, the
location result is similar to the one in Figure
\ref{fig:noise-freq_qualitative}. However, from a practical point
of view, using simultaneously $N$ different frequencies yields a
faster robust location procedure than repeating $N$ times the data
acquisition procedure with the same frequency. In subsection
\ref{subsectcharcat}, we also identify the more fundamental role
of the values of the frequencies in the characterization
procedure.

\begin{figure}
\centering%
\includegraphics[width=7cm]{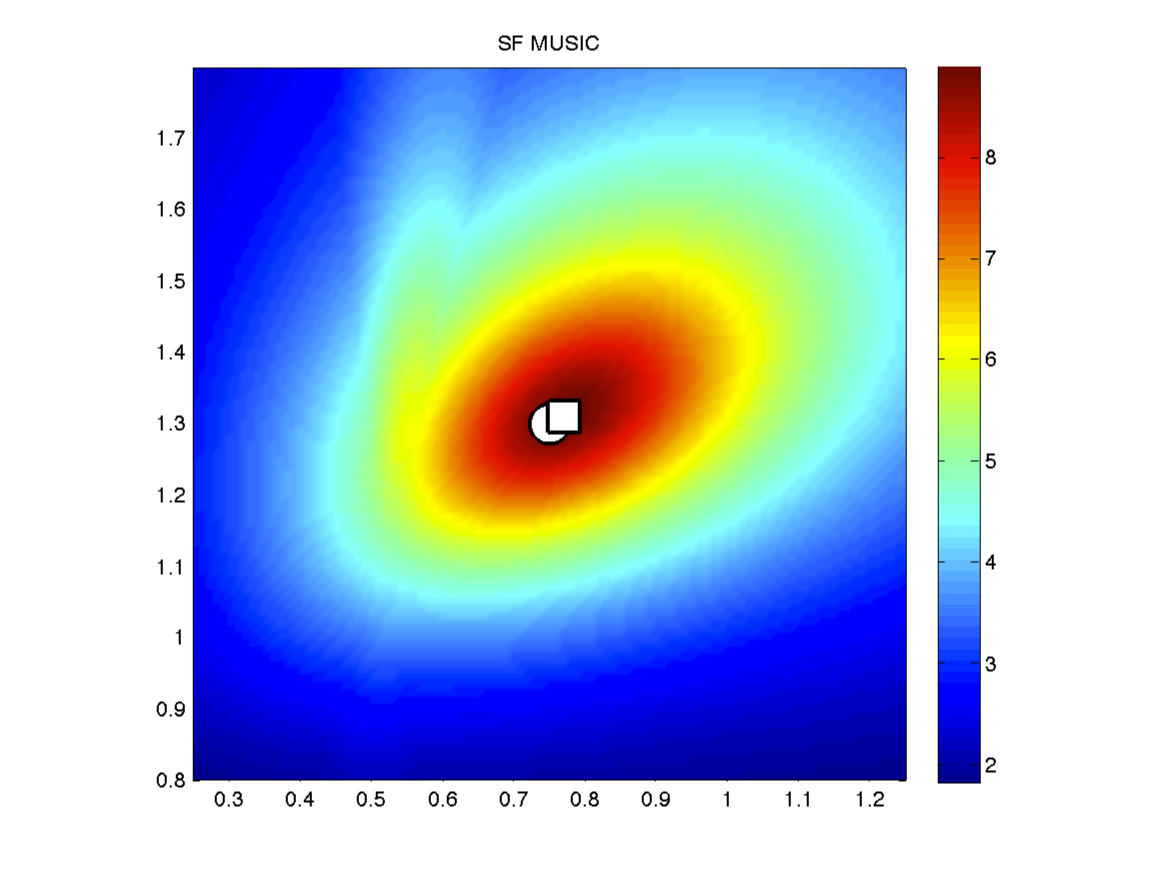}
\caption{\label{fig:noise-freq_100} Influence of the values of
used frequencies on the stability. Here, the same target as in
Figure \ref{fig:no-noise} is imaged using the data obtained by
$100$ trials with $1\%$ of noise, $64$ sensors, and frequency
equal to $1$. The disks plot the exact position, and the squares
plot the location of the maximum of the imaging functional.}
\end{figure}

The number of sensors is also crucial in the stability of the
algorithm. Figure \ref{fig:stats-sensors-noise} compares the root
mean square location error with $100$ frequencies equidistributed
from $1$ to $100$ for $64$ and $8$ sensors for different
measurement noise levels.

\begin{figure}
\centering\includegraphics[height=8cm]{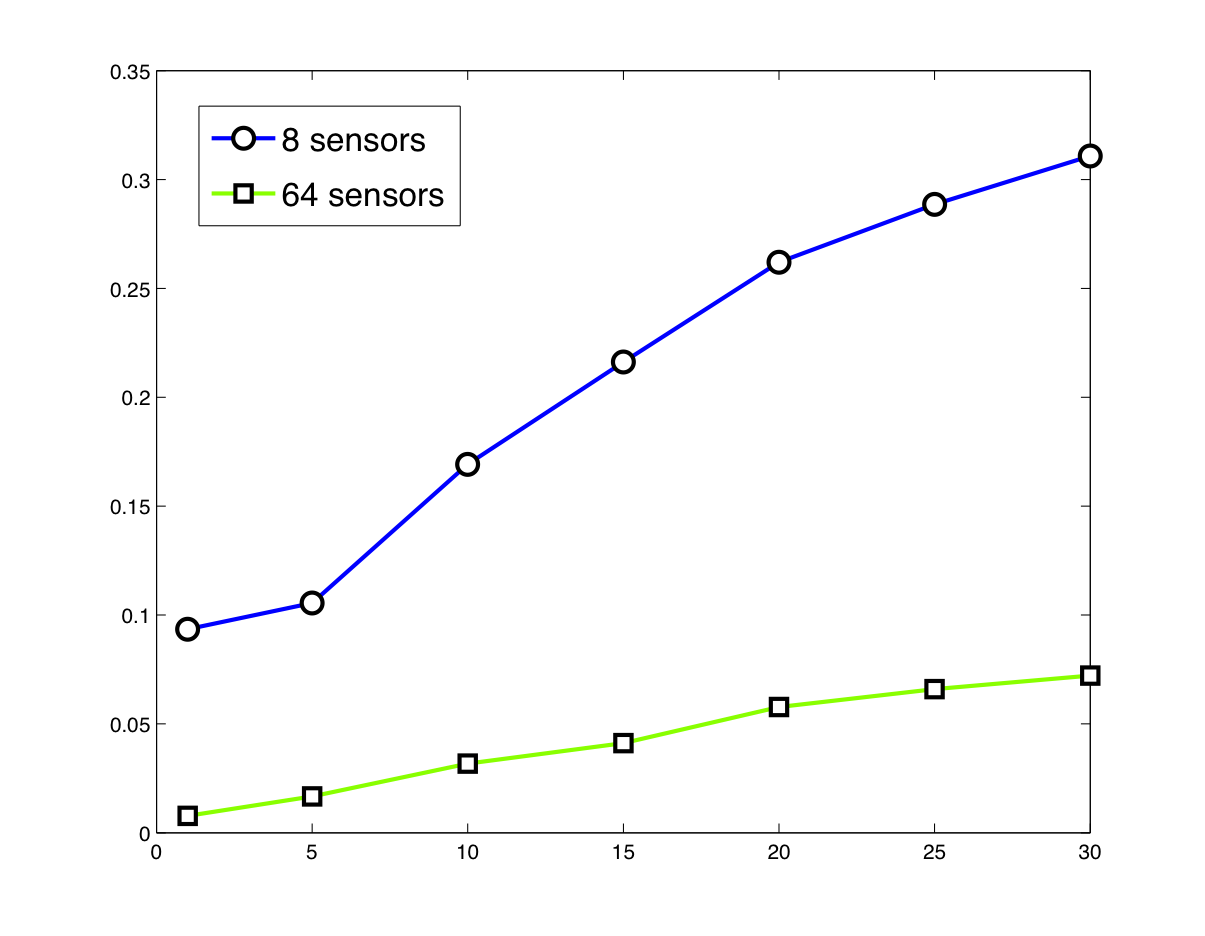}\caption{\label{fig:stats-sensors-noise}Influence
of the number of sensors on the root mean square location error
for $250$ trials. Here, the horizontal axis is for the noise level
in $\%$ and the vertical axis is for the root mean square location
error.}
\end{figure}

The same type of statistics is possible for the detection as
function of the distance between the fish from the target. In
Figure~\ref{fig:distance-noise}, we have plotted the root mean
square location errors, with $15$ frequencies equidistributed from
$1$ to $15$  and $5\%$ of noise, for disks with radius $0.05$
placed at $(t\cos(\pi/3),t\sin(\pi/3))$ for $t=1, 1.5, 2, 2.5,$
and $3$.

\begin{figure}

\centering\includegraphics[width=10.5cm]{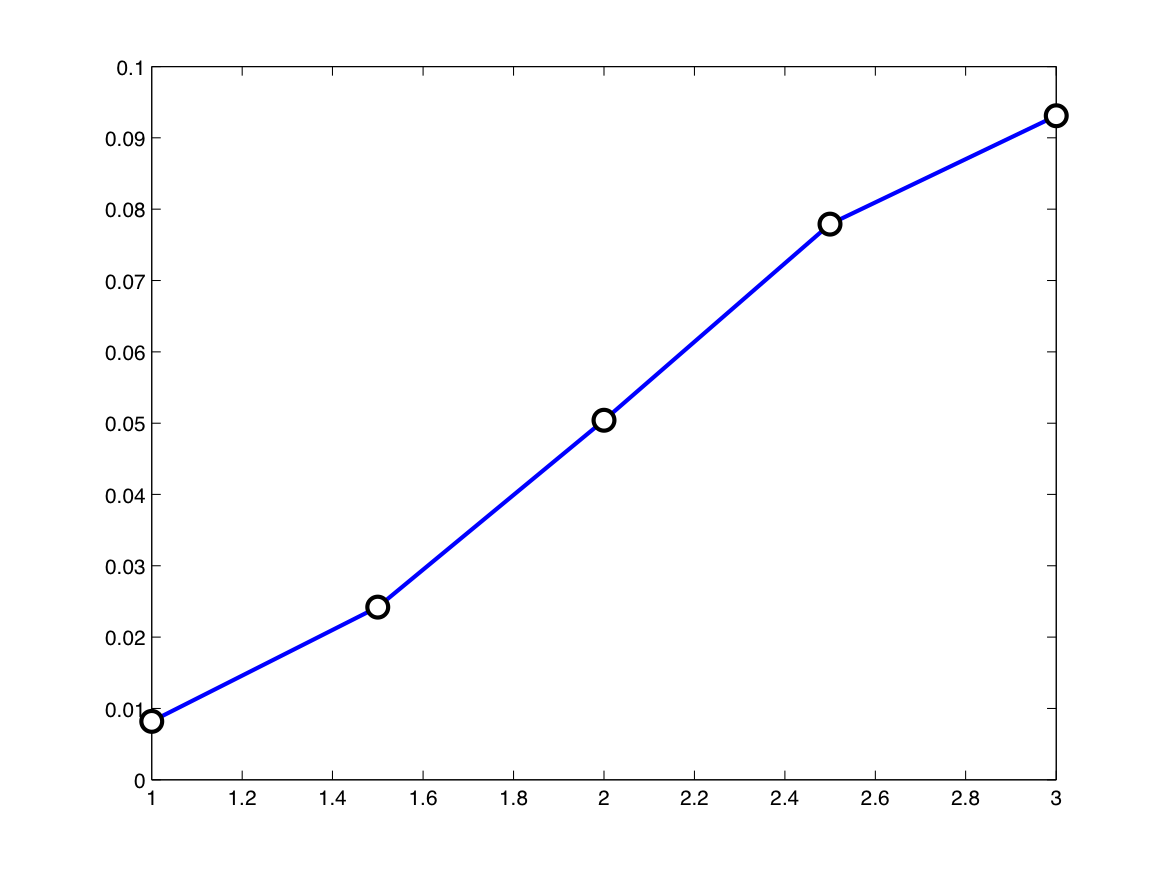}\caption{\label{fig:distance-noise}Influence
of the distance to the fish on the mean square location error for
$250$ trials. Here, the horizontal axis is for the distance to the
fish and the vertical axis is for the root mean square location
error.}

\end{figure}





\subsection{Target characterization} \label{subsectcharcat}

Once the target is located, one can use (\ref{eq:SFR-final}) to
estimate the electromagnetic parameters and the size of the
target. Assume that the target is a disk of radius $\alpha$,
placed at $z$. From (\ref{eq:SFR-final}) it follows that
$\alpha^2(k_n-1)/(k_n+1)$ can be estimated for $ 1\leq n \leq N$
from the measurement matrix $A$. Here, $k_n = k+i\varepsilon
\omega_0 n$ with $\omega_0$ being known. Let $\tau_n^{\mbox
{est}}$ be the estimated values of $\alpha^2(k_n-1)/(k_n+1)$ from
$A$. To characterize the target and approximate its size, one
minimizes the following quadratic misfit functional:
\begin{equation} \label{minimiz1}
\sum_{1 \leq n\leq N} \bigg| \frac{\alpha^2(k_n-1)}{k_n+1} -
\tau_n^{\mbox {est}} \bigg|^2,
\end{equation}
over $k, \varepsilon,$ and $\alpha$.

Table \ref{table2} gives the result of the optimization algorithm
for a disk-shaped target with center center
$(1.5\cos(\pi/3),1.5\sin(\pi/3))$ and radius $\alpha^{{\rm
true}}$. The electromagnetic parameters are $(\sigma^{{\rm true}},
\varepsilon^{{\rm true}})$. The initial guess is $\alpha^{{\rm
init}} = 0.01, \sigma^{{\rm init}} = 1, \varepsilon^{{\rm init}}
=1.$ The data is collected for $100$ frequencies  equidistributed
from $1$ to $100$. The reconstructed results are accurate.

\begin{table}[!h]
\centering
\begin{tabular}{|c|c|c||c|c|c|}
\hline $\alpha^{{\rm true}}$ & $\sigma^{{\rm true}}$ &
$\varepsilon^{{\rm true}}$ & $\alpha^{{\rm est}}$ & $\sigma^{{\rm
est}}$ & $\varepsilon^{{\rm est}}$\tabularnewline \hline 0.05 & 5
& 1 & 0.0506 & 4.9882 & 1.0004\tabularnewline \hline 0.05 & 4 & 1
& 0.0506 & 3.9993 & 0.9998\tabularnewline \hline 0.05 & 5 & 2 &
0.0506 & 4.9868 & 2.0017\tabularnewline \hline 0.06 & 5 & 1 &
0.0607 & 4.9878 & 1.0003\tabularnewline \hline 0.04 & 3 & 2 &
0.0404 & 2.9614 & 1.9806\tabularnewline \hline
\end{tabular}

\caption{Target characterization by minimizing the quadratic
misfit functional (\ref{minimiz1}) using data collected for $100$
frequencies equidistributed from $1$ to $100$. Here, ${\rm true}$:
true values, ${\rm est}$: estimated values. The initial values are
$\alpha^{{\rm init}} = 0.01, \sigma^{{\rm init}} = 1,
\varepsilon^{{\rm init}} =1.$ \label{table2}}
\end{table}

When the target is an ellipse, the measurement matrix $A$ may not
be sufficient to characterize the electromagnetic parameters and
the size of the target. At least two different positions of the
fish (or equivalently two different locations of the target in the
fish frame of reference) are needed in order to generate
non-parallel dipole directions $\nabla U/|\nabla U|$ at the
location $z$ of the target and consequently lead to the extraction
of the polarization tensor $M(k_n, D)$ of the ellipse-shaped
target $D$. Consider two target locations $z_{1}$ and $z_{2}$ in
the fish frame of reference. Multi-frequency measurements lead to
two SFR matrices, $A_{ln}^{(1)}$ and $A_{ln}^{(2)}$ with $1\leq
l\leq L$ and $1\leq n\leq N$. Define the following linear
application from the set $\mathcal{M}$ of complex symmetric
$2\times 2$ matrices to $\mathbb{C}^{2N}$
\[
F: M \mapsto\left(\begin{array}{c} \nabla
U(z_{1})^T  M  \nabla_{z}\left(\left.\frac{\partial G}{\partial\nu_{x}}\right|_{+}\right)(x_{1},z_{1})\\
\vdots\\
\nabla U(z_{1})^T M  \nabla_{z}\left(\left.\frac{\partial G}{\partial\nu_{x}}\right|_{+}\right)(x_{L},z_{1})\\
\nabla U(z_{2})^T M \nabla_{z}\left(\left.\frac{\partial G}{\partial\nu_{x}}\right|_{+}\right)(x_{1},z_{2})\\
\vdots\\
\nabla U(z_{2})^T M \nabla_{z}\left(\left.\frac{\partial
G}{\partial\nu_{x}}\right|_{+}\right)(x_{L},z_{2})
\end{array}\right)
.
\]
For a fixed $n$, we define the data
\[
b_{n}:=\left(\begin{array}{c}
A_{1n}^{(1)}\\
\vdots\\
A_{Ln}^{(1)}\\
A_{1n}^{(2)}\\
\vdots\\
A_{Ln}^{(2)}
\end{array}\right).
\]
By a least-squares method, we recover an estimation of the
polarization tensor $M(k_n,D)$:
\[
M_{n}^{{\rm est}}:=\arg\min_{M \in \mathcal{M}}\left\Vert F(M)
-b_{n}\right\Vert .
\]
Again, once $M(k_n,D)$ is estimated, a minimization approach
yields correct parameter and size values. Since for any $n$, the
eigenvectors of the matrix $M(k_n,D)$ are the ellipse axes,
denoting $\tau_{n,1}^{\mbox {est}}$ and $\tau_{n,2}^{\mbox {est}}$
the estimated complex eigenvalues of $M(k_n,D)$, one minimizes the
following quadratic misfit functional
$$
\sum_{1\leq n\leq N} \bigg| \frac{a b (k_n-1) (a+b)}{a k_n+b} -
\tau_{n,1}^{\mbox {est}} \bigg|^2 + \bigg| \frac{a b (k_n-1)
(a+b)}{b k_n+a} - \tau_{n,2}^{\mbox {est}} \bigg|^2,
$$
over $a, b, k,$ and  $\varepsilon$, in order to reconstruct the
semi-axis lengths $a$ and $b$ and the material parameters $k$ and
$\varepsilon$ of the ellipse-shaped target $D$.

If $N$ is large enough, then semi-analytical formulas to estimate
the semi-axis lengths $a,b$ and the material parameters $k,
\varepsilon$ hold. Since
$$ \tau_{N,1}^{\mbox {est}} \approx \pi a(a+b),\quad  \tau_{N,2}^{\mbox {est}} \approx \pi
b(a+b),$$ one can estimate $a$ and $b$ as follows:
\begin{equation}
a^{\rm est}  =\frac{\tau_{N,1}^{\mbox
{est}}}{\sqrt{\pi\left(\tau_{N,1}^{\mbox {est}} +\tau_{N,2}^{\mbox
{est}}\right)}},\quad  b^{\rm est}  =\frac{\tau_{N,2}^{\mbox
{est}}}{\sqrt{\pi\left(\tau_{N,1}^{\mbox {est}}+\tau_{N,2}^{\mbox
{est}}\right)}}. \label{eq:estimation_parameters}
\end{equation}
Table~\ref{tab:characteriaztion_geometric_parameters} gives
 estimations of $a$ and $b$. The target is centered at $z_{1}=1.5(\cos(\pi/3),\sin(\pi/3))$
and the fish moves in the horizontal axis so that
$z_{2}=(1.5\cos(\pi/3)-1,1.5\sin(\pi/3))$. The material parameters
of the target are $k=2$ and $\varepsilon=1$.  The data is
collected for $10$ frequencies equidistributed from $1$ to $10$.
The reconstructed results are accurate.

\begin{table}[!h]
\centering%
\begin{tabular}{|c|c||c|c|}
\hline $a^{\rm true}$ & $b^{\rm true}$ & $a^{\rm est}$ & $b^{\rm
est}$\tabularnewline \hline \hline 0.04 & 0.04 & 0.0390 &
0.0405\tabularnewline \hline 0.05 & 0.05 & 0.0497 &
0.0516\tabularnewline \hline 0.05 & 0.06 & 0.0586 &
0.0608\tabularnewline \hline \hline 0.03 & 0.06 & 0.0313 &
0.0567\tabularnewline \hline 0.06 & 0.05 & 0.0406 &
0.0487\tabularnewline \hline 0.01 & 0.03 & 0.0108 &
0.0273\tabularnewline \hline
\end{tabular}
\caption{Estimations of the semi-axis lengths of ellipse-shaped
targets using
(\ref{eq:estimation_parameters}).\label{tab:characteriaztion_geometric_parameters}}
\end{table}

Moreover, once the geometric parameters $a$ and $b$ are estimated,
it is straightforward to recover $k$ and $\varepsilon$. Introduce
\[
\begin{alignedat}{1}\mu_{n}^{(1)} & :=\frac{\tau_{N,1}^{\mbox {est}}}{\pi ab(a+b)}=\frac{k_{n}-1}{a+k_{n}b}.\end{alignedat}
\]
From
\[
k_{n} = k+ i \varepsilon n \omega_0 =
\frac{1+a\mu_{n}^{(1)}}{1-b\mu_{n}^{(1)}},
\]
one can estimate $k$ and $\varepsilon$ as the real and imaginary
parts of $k_n$. However, as shown in
Figure~\ref{fig:param_phys_est}, one can see that the error on the
real part is growing with the frequency. Therefore, in order to
increase the robustness of the material parameter estimations, one
estimate $k$ using the lowest frequencies (for example the first
three) and $\varepsilon$ using all the frequencies:
\begin{equation}
k^{\rm est} :=\frac{1}{3}\sum_{n=1}^{3}\Re\left(\frac{1+a^{\rm
est}\mu_{n}^{(1)}}{1-b^{\rm est}\mu_{n}^{(1)}}\right),\quad
\varepsilon^{\rm est}
:=\frac{1}{N}\sum_{n=1}^{N}\frac{1}{\omega_{0}n}\Im\left(\frac{1+a^{\rm
est}\mu_{n}^{(1)}}{1-b^{\rm est}\mu_{n}^{(1)}}\right).
\label{eq:phys_param_est}
\end{equation}

\begin{figure}[!h]
\centering\includegraphics[width=8cm]{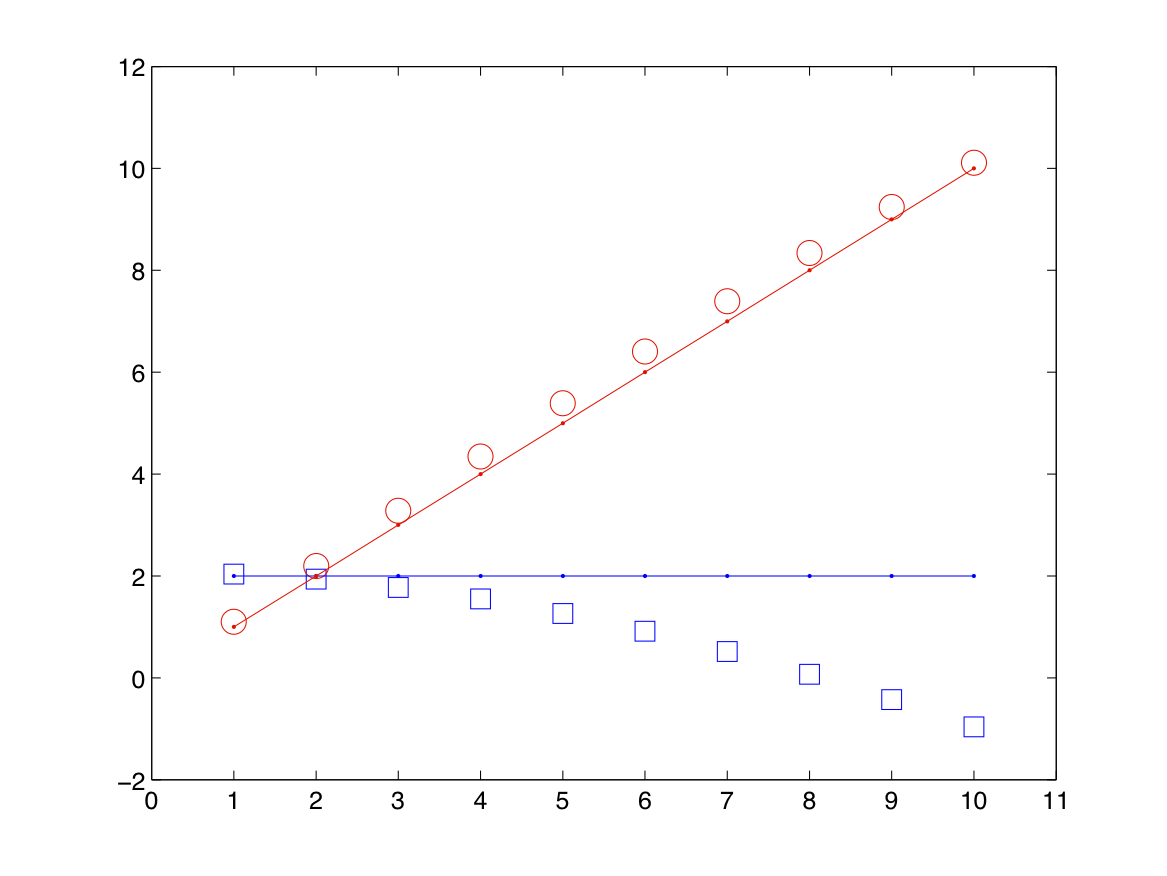}
\caption{\label{fig:param_phys_est}Real and imaginary parts
(respectively represented by squares and circles) for a
disk-shaped target as functions of the frequency. Here, the target
is with material parameters $k=2$ and $\varepsilon=1$, radius
$0,05$, and placed at $z_{1}=1.5(\cos(\pi/3),\sin(\pi/3))$ and
then at $z_{2}=(1.5\cos(\pi/3)-1,1.5\sin(\pi/3))$. The solid lines
are the theoretical values.}
\end{figure}

Table~\ref{tab:param_phys_est} gives the material estimations
using formula (\ref{eq:phys_param_est}) for a disk and an ellipse.
Once again the results are accurate.

\begin{table}[!h]
\centering%
\begin{tabular}{|c||c|c||c|c|}
\cline{2-5} \multicolumn{1}{c|}{} & $k^{\rm true}$ &
$\varepsilon^{\rm true}$ & $k^{\rm est}$ & $\varepsilon^{\rm
est}$\tabularnewline \hline
 & 2 & 1 & 1.9167 & 1.0661\tabularnewline
\cline{2-5} disk & 3 & 2 & 2.8481 & 2.0516\tabularnewline
\cline{2-5}
 & 5 & 1 & 5.8884 & 1.4668\tabularnewline
\hline \hline
 & 2 & 1 & 1.7943 & 1.0473\tabularnewline
\cline{2-5} ellipse & 3 & 2 & 2.7208 & 2.0415\tabularnewline
\cline{2-5}
 & 5 & 1 & 6.0886 & 1.5828\tabularnewline
\hline
\end{tabular}

\caption{\label{tab:param_phys_est}Estimations of the material
parameters based on formula (\ref{eq:phys_param_est}). The disk
has radius $0,05$ and the ellipse has semi-axis lengths $0,025$
and $0,1$ and orientation angle $\pi/3$. Both targets are placed
at $z_{1}=1.5(\cos(\pi/3),\sin(\pi/3))$ and then at
$z_{2}=(1.5\cos(\pi/3)-1,1.5\sin(\pi/3))$, and are illuminated
with $10$ frequencies equidistributed from $1$ to $10$.}
\end{table}

\section{Conclusion}

In this paper, we have proposed a complex conductivity model
problem for the quantitative analysis of active electro-location
in weakly electric fish. We have rigorously derived the boundary
conditions to be used. We have proposed a non-iterative location
search algorithm based on multi-frequency measurements.  We have
presented some numerical results which are promising. We have seen
that increasing the number of frequencies (with not necessary
different values) improves the stability. In fact, using
multi-frequency measurements increases the signal-to-noise ratio.
On the other hand, using different frequencies yields a faster
robust location algorithm than repeating the data acquisition
procedure with the same frequency. We have also proposed a
procedure to reconstruct the electromagnetic parameters and the
size of disk- and ellipse-shaped targets. This has been possible
only because of multi-frequency measurements corresponding here to
different frequency values. The use of multi-frequency
measurements is fundamental in the characterization procedure. It
has been known that polarization tensor for real conductivities
cannot separate the size from material properties of the target
\cite{ammari2007polarization}. For arbitrary-shaped targets, many
important questions remain. In particular, it would be interesting
to know how much parameter and size information one can extract
from its polarization tensors for different complex
conductivities. It is also worth mentioning that limiting our
asymptotic expansions with respect to the target size to the
first-order term (the dipole approximation) does not give us the
shape of the target. Hence, in a forthcoming work we will
investigate how much information can be acquired in the near field
by approaching the fish next to the target and developing the
asymptotic expansions with high-order generalized polarization
tensors \cite{ammari2004reconstruction}. We will also investigate
the stability of the proposed algorithm with respect to random
fluctuations in the background permittivity and propose an
original cross-correlation technique in order to correct for the
effect of random heterogeneities on target location.

\bibliographystyle{plain} \bibliographystyle{plain} \bibliographystyle{plain}
\bibliography{biblio}

\end{document}